\documentclass[12pt]{amsart}
\usepackage{graphicx}
\usepackage{amscd}
\usepackage{amsmath}
\usepackage{amssymb}
\usepackage{amsfonts}
\usepackage[margin=1.2in]{geometry}

\newtheorem{theorem}{Theorem}[section]
\theoremstyle{plain}
\newtheorem{corollary}{Corollary}[section]

\newtheorem{proposition}{Proposition}[section]

\theoremstyle{definition}
\newtheorem{example}{Example}[section]
\newtheorem{definition}{Definition}[section]

\numberwithin{equation}{section}

\begin{document}
\title[Class estimates for vector-valued distributions]{Tauberian class estimates for vector-valued distributions}
\author[S. Pilipovi\'{c}]{Stevan Pilipovi\'{c}}
\address{Department of Mathematics and Informatics\\ University of Novi Sad\\ Trg Dositeja Obradovi\'ca 4\\ 21000 Novi Sad \\ Serbia }
\email{pilipovic@dmi.uns.ac.rs}

\thanks{S. Pilipovi\'{c} gratefully acknowledges support by Ministry of Edu. Sci. Tech. Dev. grant 174024.}

\author[J. Vindas]{Jasson Vindas}
\address{Department of Mathematics: Analysis, Logic and Discrete Mathematics, Ghent University, Krijgslaan 281, 9000 Gent, Belgium}
\email{jasson.vindas@UGent.be}

\thanks{The work of J. Vindas was supported by Ghent University, through the BOF-grant 01N01014.}



\begin{abstract}
We study Tauberian properties of regularizing transforms of vector-valued tempered distributions, that
is, transforms of the form $M^{\mathbf{f}}_{\varphi}(x,y)=(\mathbf{f}\ast\varphi_{y})(x)$, where the kernel $\varphi$ is a test
function and $\varphi_{y}(\cdot)=y^{-n}\varphi(\cdot/y)$. We investigate conditions which ensure that a distribution that a priori takes values in locally convex space actually takes values in a narrower Banach space. Our goal is to characterize spaces of Banach space valued tempered distributions in terms of so-called \emph{class estimates} for the transform $M^{\mathbf{f}}_{\varphi}(x,y)$. Our results generalize and improve earlier Tauberian theorems of Drozhzhinov and Zav'yalov (Sb. Math. 194 (2003), 1599--1646). Special attention is paid to find the \emph{optimal} class of kernels $\varphi$ for which these Tauberian results hold. 

\end{abstract}

\maketitle

\begin{center}

\emph{Dedicated to the memory of Vasili\u{\i} Sergeevich Vladimirov\\ and Boris  Ivanovich Zav'yalov}

\end{center}

\bigskip

\section{Introduction}
\label{wnwi}

Tauberian theorems are quite useful tools in several fields of mathematics such as number theory, operator theory, differential equations, probability theory, and mathematical physics. See Korevaar's book \cite{korevaarbook} for an account on the one dimensional theory. In the case of multidimensional Tauberian theorems, the subject has been deeply influenced by the extensive work of Drozhzhinov, Vladimirov, and Zav'yalov. Their approach led to the incorporation of generalized functions in the area and resulted in a powerful Tauberian machinery for multidimensional Laplace transforms of distributions. We refer to the monographs \cite{P-S-V,vladimirov-d-z1} and the recent survey article \cite{Drozhzhinov2016} for overviews on Tauberian theorems for generalized functions and their applications. Interestingly, distributional methods have been crucial for recent developments on complex Tauberian theorems \cite{C,D-V2016,D-V_OPT,D-V_CT,korevaar2005}.

The goal of this article is to generalize and improve various of the results of Drozhzhinov and Zav'yalov from \cite{drozhzhinov-z2,drozhzhinov-z3}. The present work might be regarded as a continuation of our previous paper \cite{Pil-Vin2014}, where we treated multidimensional Tauberian theorems for quasiasymptotics of vector-valued distributions (see also \cite{vindas-pilipovic-rakic}). We shall deal here with characterizations of Banach space valued distributions in terms of so-called Tauberian \emph{class estimates}. 

The central problem that we shall be considering can be formulated as follows. Fix a test function $\varphi\in\mathcal{S}(\mathbb{R}^{n})$. To vector-valued tempered distribution $\mathbf{f}$, we associate the \emph{regularizing transform}
$$
M^{\mathbf{f}}_{\varphi}(x,y)= (\mathbf{f}\ast\varphi_{y})(x), \qquad (x,y)\in\mathbb{R}^{n}\times \mathbb{R}_{+},$$
where $\varphi_{y}(t)=y^{-n}\varphi(t/y)$. We are interested in conditions in terms of $M^{\mathbf{f}}_{\varphi}$ that ensure that $\mathbf{f}$ a priori taking values in a ``broad'' (Hausdorff) locally convex space $X$ actually takes values in a narrower Banach space $E$, which is assumed to be continuously included into $X$. If it is a priori known that $\mathbf{f}$ takes values in $E$, then one can directly verify that it satisfies the norm estimate
\begin{equation}
\label{wnwieq4} 
\left\|M^{\mathbf{f}}_{\varphi}(x,y)\right\|_{E}\leq
C \frac{(1+y)^{k}\left(1+\left|x\right|\right)^{l}}{y^{k}},
\end{equation}
for some $k$, $l$, and $C$. As in \cite{drozhzhinov-z2,drozhzhinov-z3}, we call (\ref{wnwieq4}) a class estimate. The problem of interest is thus the converse one: Up to what extend does a class estimate (\ref{wnwieq4}) allow one to conclude that $\mathbf{f}$ takes values in the Banach space $E$? 

This Tauberian question was raised and studied by Drozhzhinov and Zav'yalov in one dimension in \cite{drozhzhinov-z2} and in several variables in \cite{drozhzhinov-z3}, where they demonstrated that many Tauberian theorems for integral transforms become particular instances of this problem. See also \cite{drozhzhinov-z5} for several other interesting applications. In this article we revisit the problem and provide optimal results by finding the largest class of test functions for which the space of $E$-valued distributions (up to some correction terms) admits a characterization in terms of a class estimate.

We will essentially prove here that the sought optimal Tauberian kernels are given by the class of non-degenerate test functions that we introduced in \cite{Pil-Vin2014} and turns out to be much larger than the one employed in \cite{drozhzhinov-z2,drozhzhinov-z3} (the latter called in this article the class of strongly non-degenerate test functions, see  Definition \ref{wnwced2}). Note that in Wiener Tauberian theory \cite{korevaarbook} the Tauberian kernels are those whose Fourier transforms do not vanish at any point. In our theory the Tauberian kernels will be those $\varphi$ such that $\hat{\varphi}$ does not identically vanish on any ray through the origin. Besides yielding more general results, we believe that our new approach, based on ideas from wavelet analysis, is much simpler than that from \cite{drozhzhinov-z2,drozhzhinov-z3}, as it totally avoids using ``corona theorem'' type arguments and the structure of the Taylor polynomials of the Fourier transforms $\hat{\varphi}$ of the kernels.

 The article is organized as follows. Section \ref{prely wnwn} fixes the notation and collects some necessary background material on wavelet analysis of Banach space valued (Lizorkin) distributions. Section \ref{wnwtd} is also of preparatory character, we discuss there some basic properties and examples of regularizing transforms. The main body of the article is Sections \ref{wnwce}--\ref{section class estimates with representation}. Section \ref{wnwce} deals with characterizations of $\mathcal{S}'(\mathbb{R}^{n},E)$ in terms of global and local class estimates, that is, when \eqref{wnwieq4} holds for all $(x,y)$ or when it is just assumed for $(x,y)\in\mathbb{R}^{n}\times (0,1]$. We shall actually work there with more general integral assumptions on $M^{\mathbf{f}}_{\varphi}$. We analyse in Section \ref{wnwce} the case of strongly non-degenerate test functions as kernels of the regularizing transform. Finally,  we study in Section \ref{section class estimates with representation} Tauberian class estimate characterizations for distributions intertwining a representation of $\mathbb{R}^{n}$ with the translation group; the underlying conditions here are given in terms of  transforms with respect to (generalized) Littlewood-Paley pairs.

\section{Preliminaries}
\label{prely wnwn}
We use the notation $\mathbb{H}^{n+1} = \mathbb{R}^{n} \times
\mathbb{R_+}$ for the upper-half space. Locally convex spaces are always assumed to be Hausdorff. The space $E$ always denotes a fixed, but arbitrary, Banach space with norm $\left\|\:\cdot\:\right\|$. Measurability for $E$-valued functions is
meant in the sense of Bochner (i.e., a.e. pointwise limits of
$E$-valued continuous functions); likewise, integrals for
$E$-valued functions are taken in the Bochner sense. For test functions we set $\check{\varphi}(t)=\varphi(-t)$ and $\varphi_{y}(t)=y^{-n}\varphi( t/y)$.

\subsection{Spaces of test functions} \label{wnw sub spaces}We use the standard notation from distribution theory, as explained e.g. in \cite{P-S-V, vladimirovbook,vladimirov-d-z1}. In particular, the Schwartz spaces of smooth compactly supported and rapidly decreasing test functions are denoted by $\mathcal{D}(\mathbb{R}^{n})$ and $\mathcal{S}(\mathbb{R}^{n})$.   We choose the constants in the Fourier transform as
$$\hat{\varphi}(u)=\int_{\mathbb{R}^{n}}\varphi(t)e^{-iu \cdot t}\mathrm{d}t.$$

 Following
\cite{holschneider}, the space
$\mathcal{S}_0(\mathbb{R}^{n})$ of highly time-frequency localized
functions over $\mathbb{R}^{n}$ is defined as the closed subspace of $\mathcal{S}(\mathbb{R}^{n})$ consisting of  those elements for which all their moments vanish,
i.e., $$\eta\in \mathcal{S}_0(\mathbb{R}^{n})\; \mbox{ if and only
if }\;  \int_{\mathbb{R}^{n}}t^m\eta(t)\mathrm{d}t=0, \;\mbox{ for
all }\; m\in\mathbb{N}^{n}.$$ This
space is also known as the Lizorkin space of test functions. The corresponding space of highly
localized function over $\mathbb H^{n+1}$ is denoted by
$\mathcal{S}(\mathbb {H}^{n+1})$. It consists of those $\Phi\in
C^{\infty}(\mathbb{H}^{n+1})$ for which
$$\sup_{(x,y)\in
\mathbb {H}^{n+1}}\,\left(y+\frac
{1}{y}\right)^{k_{1}}\left(1+\left|x\right|\right)^{k_{2}}\,\left|\frac{\partial^{l}}{\partial y^{l}}\frac{\partial^{m}}{\partial x^{m}}\Phi (x,y)\right|<\infty,$$
for all $k_{1},k_{2},l\in \mathbb{N}$ and $m\in\mathbb{N}^{n}$.
The canonical topology of this space is defined in the standard way \cite{holschneider}.

We shall also employ an interesting class of subspaces of $\mathcal{S}(\mathbb{R}^{n})$ introduced by Drozhzhinov and Zav'yalov in \cite{drozhzhinov-z3}. Let $I$ be an ideal of the ring $\mathbb{C}[t_1,t_{2},\dots,t_{n}]$, the (scalar-valued) polynomials over $\mathbb{C}$ in $n$ variables. Define $\mathcal{S}_{I}(\mathbb{R}^{n})$ as the subspace of $\mathcal{S}(\mathbb{R}^{n})$ consisting of those $\phi$ such that all Taylor polynomials at the origin of its Fourier transform $\hat{\phi}$ belong to the ideal $I$. For example, we have $\mathcal{S}_{I}(\mathbb{R}^{n})=\mathcal{S}_{0}(\mathbb{R}^{n})$ if $I=\left\{0\right\}$ or $\mathcal{S}_{I}(\mathbb{R}^{n})=\mathcal{S}(\mathbb{R}^{n})$ if $I=\mathbb{C}[t_1,t_{2},\dots,t_{n}]$. 

Let $P_{0},\dots,P_{q},\dots$ be a system of homogeneous polynomials where each $P_{q}$ has degree $q$ (some of them may be identically 0). Consider the ideal $I=[P_{0},P_{1},\dots, P_{q},\dots]$, namely, the ideal generated by the $P_{q}$;  then, one can show \cite{drozhzhinov-z3}  that $\mathcal{S}_{I}(\mathbb{R}^{n})$ is a closed subspace of $\mathcal{S}(\mathbb{R}^{n})$ and actually $\phi\in\mathcal{S}_{I}(\mathbb{R}^{n})$ if and only if 
$$\int_{\mathbb{R}^{n}} Q(t)\phi(t)\mathrm{d}t =0$$
for all polynomial $Q$ that satisfies the differential equations 
$$P_{q}(\partial/\partial t)Q=0,\qquad q=0,1\dots .$$ 
When there is $d\in\mathbb{N}$ such that $P_{q}=0$ for $q>d$ one can relax the previous requirement \cite[Lem. A.5]{drozhzhinov-z3} by just asking it to hold for polynomials $Q$ with degree at most $d$. 

We denote by $\mathbb{P}_{d}$ the ideal of (scalar-valued) polynomials of the form $Q(t)=\sum_{d\leq\left|m\right|\leq N}a_{m}t^{m}$, for some $N\in\mathbb{N}$.

Let $\varphi\in\mathcal{S}(\mathbb{R}^{n})$. We write $P^{\varphi}_{q}$ for the $q$-homogeneous term of the Taylor polynomial of its Fourier transform at 0, that is,
\begin{equation}
\label{taylorhomoeq} 
P^{\varphi}_{q}(u)=\sum_{|m|=q} \frac{\hat{\varphi}^{(m)}(0)}{m!} u^{m} , \qquad q=0,1,2,\dots.
\end{equation} 
The ideal generated by these homogeneous polynomials is denoted as 
\begin{equation}
\label{classestIdealeq}
I_{\varphi}=[P_{0}^{\varphi},P_{1}^{\varphi},\dots, P_{q}^{\varphi},\dots].
\end{equation}

\subsection{Spaces of vector-valued distributions} 
\label{wnwdsp} 
Let $\mathcal{A}(\Omega)$ be a topological vector space of test function over
 an open subset $\Omega\subseteq\mathbb{R}^{n}$ and let $X$ be a locally convex spaces.
 We denote by $\mathcal{A}'(\Omega,X)=L_{b}(\mathcal{A}(\Omega),X)$, the space of continuous
 linear mappings from $\mathcal{A}(\Omega)$ to $X$ with the topology of uniform convergence over
 bounded subsets of $\mathcal{A}(\Omega)$ \cite{treves}. We are mainly concerned with
 the spaces $\mathcal{D}'(\mathbb{R}^{n},X)$, $\mathcal{S}'(\mathbb{R}^{n},X)$, $\mathcal{S}'(\mathbb{H}^{n+1},X)$, and $\mathcal{S}'_{I}(\mathbb{R}^{n},X)$; see \cite{schwartzv,silva} for vector-valued
 distributions. 
 
Observe that we have a well defined continuous linear projector
from $\mathcal{S}'(\mathbb{R}^{n},X)$ onto $\mathcal{S}'_I(\mathbb{R}^{n},X)$ as
the restriction of $X$-valued tempered distributions to
$\mathcal{S}_I(\mathbb{R}^{n})$. We
do not want to introduce a notation for this map, so if
$\mathbf{f}\in\mathcal{S}'(\mathbb{R}^{n},X)$, we will keep calling by $\mathbf{f}$ its projection onto $\mathcal{S}'_I(\mathbb{R}^{n},X)$. 

\subsection{Wavelet analysis on $\mathcal{S}'_{0}(\mathbb{R}^{n},E)$} 
\label{waE} We review in this section some properties of the wavelet transform of (Lizorkin) distributions with values in the Banach space $E$. By a wavelet we simply mean an element of $\mathcal{S}_{0}(\mathbb{R}^{n})$. The wavelet transform of  $\mathbf{f}\in \mathcal{S}'_{0}(\mathbb{R}^{n},E)$ with respect to the wavelet $\psi\in\mathcal{S}_{0}(\mathbb{R}^{n})$ is defined as
$$
\mathcal{W}_{\psi}\mathbf{f}(x,y)=\langle \mathbf{f}(x+yt), \bar{\psi}(t) \rangle \in E, \qquad (x,y)\in \mathbb{H}^{n+1}.
$$
It is then clear that $\mathcal{W}_{\psi}\mathbf{f} \in C^{\infty}(\mathbb{H}^{n+1}, E)$. 

Notice \cite{holschneider,P-R-V} that $\mathcal{W}_{\psi}:\mathcal{S}_{0}(\mathbb{R}^{n})\mapsto\mathcal{S}(\mathbb{H}^{n+1})$ is a continuous linear map. We are interested in those wavelets for which $\mathcal{W}_{\psi}$ admits a left inverse. For wavelet-based reconstruction, one employs the so-called wavelet
synthesis operator \cite{holschneider}. Given $\Phi\in\mathcal{S}(\mathbb{H}^{n+1})$, we define the \textit{wavelet
synthesis operator} with respect to the wavelet $\psi$ as
\begin{equation}
\label{wnwneq6}
\mathcal{M}_{\psi}\Phi(t)=\int^{\infty}_{0}\int_{\mathbb{R}^{n}}\Phi(x,y)\frac{1}{y^{n}}\psi\left(\frac{t-x}{y}\right)\frac{\mathrm{d}x\mathrm{d}y}{y}\: ,
\ \ \ t \in \mathbb{R}^{n}.
\end{equation}
One can show that $\mathcal{M}_{\psi}:\mathcal{S}(\mathbb{H}^{n+1})\rightarrow
\mathcal{S}_0(\mathbb{R}^{n})$ is continuous \cite{holschneider,vindas-pilipovic-rakic}.

We shall say that the wavelet $\psi\in\mathcal{S}_{0}(\mathbb{R}^{n})$ admits a \textit{reconstruction wavelet} if there exists $\eta\in\mathcal{S}_{0}(\mathbb{R}^{n})$ such that
\begin{equation}
\label{wnwneq7}
c_{\psi,\eta}(\omega)=\int^{\infty}_{0}\overline{\hat{\psi}}(r\omega)\hat{\eta}(r\omega)\frac{\mathrm{d}r}{r}\neq 0\: , \ \ \ \omega\in\mathbb{S}^{n-1} ,
\end{equation}
is independent of the direction $\omega$; in such a case we set $c_{\psi,\eta}:=c_{\psi,\eta}(\omega).$ If $\psi$ admits the reconstruction wavelet $\eta$, one has the reconstruction
formula \cite{holschneider} for the wavelet transform on $\mathcal{S}_{0}(\mathbb{R}^{n})$
\begin{equation}
\label{wnwneq8}
\mathrm{Id}_{\mathcal{S}_0(\mathbb{R}^{n})}=\frac{1}{c_{\psi,\eta}}\mathcal{M}_{\eta}\mathcal{W}_{\psi}.
\end{equation}

We have characterized in \cite[Prop. 5.1]{Pil-Vin2014} those wavelets which have a reconstruction wavelet. They are in fact those elements of $\mathcal{S}_{0}(\mathbb{R}^{n})$ that are non-degenerate in the sense of the ensuing definition.

\begin{definition}
\label{wnwd1}
We say that the test function $\varphi\in\mathcal{S}(\mathbb{R}^{n})$ is non-degenerate if its Fourier transform does not identically vanish on any ray through the origin.
\end{definition}

The wavelet synthesis operator (\ref{wnwneq6}) can be extended to $\mathcal{S}'_{0}(\mathbb{H}^{n+1},E)$ as follows. Let $\mathbf{K}\in\mathcal{S}'_{0}(\mathbb{H}^{n+1},E)$. We define $\mathcal{M}_{\psi}:\mathcal{S}'_{0}(\mathbb{H}^{n+1},E)\mapsto\mathcal{S}'_{0}(\mathbb{R}^{n},E)$, a continuous linear map, as
\begin{equation}
\label{wnwneq10}
\left\langle
\mathcal{M}_{\psi}\mathbf{K}(t),\rho(t)\right\rangle=\left\langle
\mathbf{K}(x,y),\mathcal{W}_{\bar{\psi}}\rho(x,y)\right\rangle, \ \ \ \rho\in\mathcal{S}_{0}(\mathbb{R}^{n}).
\end{equation}
We mention that our convention to identify a function $\mathbf{F}$ of slow growth on $\mathbb{H}^{n+1}$ with an element of $\mathcal{S}'_{0}(\mathbb{R}^{n},E)$, that is, one that satisfies a growth condition 
$$
\int_{0}^{\infty}\int_{\mathbb{R}^{n}} \left(\frac{1}{y}+{y}\right)^{-k}(1+|x|)^{-l} \|\mathbf{F}(x,y)\| < \infty,$$
is via the Bochner integral
$$
\left\langle
\mathbf{F}(x,y), \Phi(x,y)\right\rangle =\int_{0}^{\infty}\int_{\mathbb{R}^{n}} \Phi(x,y)\mathbf{F}(x,y) \frac{\mathrm{d}x\mathrm{d}y}{y}, \qquad \Phi\in \mathcal{S}(\mathbb{H}^{n+1}).
$$
If $\psi\in\mathcal{S}_{0}(\mathbb{R}^{n})$ is non-degenerate and $\eta\in\mathcal{S}_{0}(\mathbb{R}^{n})$ is a reconstruction wavelet for it, then \cite[Prop. 5.3]{Pil-Vin2014} the reconstruction formula
\begin{equation}
\label{wnwneq11}
\mathrm{Id}_{\mathcal{S}'_0(\mathbb{R}^{n},E)}=\frac{1}{c_{\psi,\eta}}\mathcal{M}_{\eta}\mathcal{W}_{\psi} 
\end{equation}
holds; 
furthermore, we have the desingularization formula,
\begin{equation}
\label{wnwneq12}
\left\langle
\mathbf{f}(t),\rho(t)\right\rangle=\frac{1}{c_{\psi,\eta}}\int^{\infty}_{0}\int_{\mathbb{R}^{n}}\mathcal{W}_{\psi}\mathbf{f}(x,y)\mathcal{W}_{\bar{\eta}}\rho(x,y)\frac{\mathrm{d}x\mathrm{d}y}{y}\: ,
\end{equation}
valid for all $\mathbf{f}\in\mathcal{S}'_{0}(\mathbb{R}^{n},E)$ and $\rho\in\mathcal{S}_{0}(\mathbb{R}^{n})$.

\section{Regularizing transform of $E$-valued distributions}
\label{wnwtd}
We discuss in this section some basic properties of the regularizing transform of $E$-valued tempered distributions with respect to a test function $\varphi\in\mathcal{S}(\mathbb{R}^{n})$, which, as in the introduction, we define as the $E$-valued $C^{\infty}$-function
\begin{equation}
\label{wnwneq1}
M^{\mathbf{f}}_{\varphi}(x,y):=(\mathbf{f}\ast\varphi_{y})(x), \ \ \ (x,y)\in\mathbb{H}^{n+1},
\end{equation}
where $\mathbf{f}\in\mathcal{S}'(\mathbb{R}^{n},E)$. If $\psi\in\mathcal{S}_{0}(\mathbb{R}^{n})$, we obviously have 
$$
\mathcal{W}_{\psi}\mathbf{f}=M^{\mathbf{f}}_{\check{\bar{\psi}}}.
$$ Naturally, the transform  
\eqref{wnwneq1} perfectly makes sense for distributions with values in any arbitrary locally convex space. 

The study of the (Tauberian) ``converse'' of the bound on the regularizing transform of an $E$-valued distribution delivered by the next (Abelian) proposition is the core of this article. 

\begin{proposition}
\label{wnwp2} Let $\mathbf{f}\in\mathcal{S}'(\mathbb{R}^{n},E)$
and let $\varphi\in\mathcal{S}(\mathbb{R}^{n})$. Then,
$M^{\mathbf{f}}_{\varphi}\in C^{\infty}(\mathbb{H}^{n+1},E)$ is a
function of slow growth on $\mathbb{H}^{n+1}$. In addition, the
linear map  $\mathbf{f}\in\mathcal{S}'(\mathbb{R}^{n},E)\mapsto
M^{\mathbf{f}}_{\varphi}\in\mathcal{S}'(\mathbb{H}^{n+1},E)$ is
continuous for the topologies of uniform convergence over bounded
sets. Furthermore, if
$\mathfrak{B}\subset\mathcal{S}'(\mathbb{R}^{n},E)$ is bounded for
the topology of pointwise convergence, then there exist $k,l$ and
$C>0$ such that
\begin{equation}
\label{wnwneq5}
\left\|M_{\varphi}^\mathbf{f}(x,y)\right\|\leq C \left( \frac{1}{y}+y\right)^{k}\left(1+\left|x\right|\right)^{l} , \qquad (x,y)\in \mathbb{H}^{n+1},
\end{equation}
for all $\mathbf{f}\in\mathfrak{B}.$
\end{proposition}
\begin{proof} The proof is simple but we include it for the sake of completeness.
The space $\mathcal{S}'(\mathbb{R}^{n},E)$ is bornological. Therefore, we
should show that this map takes bounded sets to bounded ones. Let
$\mathfrak{B}\subset\mathcal{S}'(\mathbb{R}^{n},E)$ be a bounded
set. The Banach-Steinhaus theorem implies the existence of
$k_{1}\in\mathbb{N}$ and $C_{1}>0$ such that for all $\rho\in \mathcal{S}(\mathbb{R}^{n})$ and $\mathbf{f}\in\mathfrak{B}$,

$$
\left\|\left\langle \mathbf{f},\rho\right\rangle\right\|\leq C_{1} \sup_{t\in\mathbb{R}^{n}, \left|m\right|\leq k_{1}} \left(1+\left|t\right|\right)^{k_1}\left|\rho^{(m)}(t)\right| .
$$
Consequently,
\begin{align*}
\left\|M^{\mathbf{f}}_{\varphi}(x,y)\right\| & 
\leq C_{1}\left(\frac{1}{y}+y\right)^{n+k_1} \sup_{u\in\mathbb{R}^{n}, \left|m\right|\leq k_1} \left(1+\left|x\right|+y\left|u\right|\right)^{k_1}\left|\varphi^{(m)}\left(u\right)\right|
\\
&
\leq
C\left(\frac{1}{y}+y\right)^{n+2k_{1}}(1+\left|x\right|)^{k_{1}}  , \ \ \mbox{ for all } \mathbf{f}\in\mathfrak{B},
\end{align*}
where $C=C_{1}\sup_{u\in\mathbb{R}^{n}, \left|m\right|\leq k_1}
\left(1+\left|u\right|\right)^{k_1}\left|\varphi^{(m)}\left(u\right)\right|.$
So, we obtain (\ref{wnwneq5}) with $k=n+2k_{1}$ and $l=k_{1}$. If
$\mathfrak{C}\subset\mathcal{S}(\mathbb{H}^{n+1})$ is a bounded
set of test functions, we have
\begin{align*}
\left\|\left\langle M^{\mathbf{f}}_{\varphi}(x,y),\Phi(x,y)\right\rangle\right\|&=\left\|\int^{\infty}_{0}\int_{\mathbb{R}^{n}}M^\mathbf{f}_{\varphi}(x,y)\Phi(x,y)\frac{\mathrm{d}x\mathrm{d}y}{y}\right\|
\\
&
\leq C \int^{\infty}_{0}\int_{\mathbb{R}^{n}}\left(\frac{1}{y}+y\right)^{k}(1+\left|x\right|)^{l}\left|\Phi(x,y)\right|\frac{\mathrm{d}x\mathrm{d}y}{y}\, ,
\end{align*}
which stays bounded as $\mathbf{f}\in\mathfrak{B}$ and $\Phi\in\mathfrak{C}$. Therefore, the set $$\left\{M^{\mathbf{f}}_{\varphi}: \; \mathbf{f}\in\mathfrak{B}\right\}\subset\mathcal{S}'(\mathbb{H}^{n+1},E)$$
is bounded; hence the map is continuous.
\end{proof}
We point out that the regularizing transforms enjoy excellent localization properties as shown by the following simple proposition. 

\begin{proposition} Let $\mathbf{f}\in\mathcal{S}'(\mathbb{R}^{n},E)$ and let $\varphi\in\mathcal{S}(\mathbb{R}^{n})$. Suppose that $K\subset \mathbb{R}^{n}\setminus \operatorname*{supp} \mathbf{f}$ is a compact set. Then, for any positive integer $k\in\mathbb{N}$ there exists $C=C_{k}$ such that
\begin{equation}
\label{wnweqLoc}
\sup_{x\in K}\left\|M^{\mathbf{f}}_{\varphi}(x,y)\right\|\leq C y^{k}, \  \ \ \mbox{for all }0<y<1.
\end{equation}
\end{proposition}
\begin{proof} Define the $C(K,E)$-valued tempered distribution whose evaluation at $\rho\in\mathcal{S}(\mathbb{R}_{t}^{n})$ is given by
$\left\langle \mathbf{G}(t),\rho(t)\right\rangle(\xi)=(\mathbf{f}\ast \rho)(\xi),$ $\xi\in K.$
Clearly, $\mathbf{G}\in\mathcal{S}'(\mathbb{R}_{t}^{n},C(K_{\xi},E))$. Then, since $K\subset \mathbb{R}^{n}\setminus \operatorname*{supp} \mathbf{f}$, we have that for each $\rho\in\mathcal{D}(\mathbb{R}^{n})$,
$
\left\langle \mathbf{G}(\varepsilon t),\rho (t)\right\rangle=0,
$
for sufficiently small $\varepsilon>0$. In particular, we obtain that, for a fixed $k\in\mathbb{N}$,
\begin{equation}
\label{wnweqpb1}
\mathbf{G}(\varepsilon t)=O(\varepsilon^{k}) \ \ \ \mbox{as }\varepsilon \to 0^{+}\ \mbox{ in }\mathcal{D}'(\mathbb{R}_{t}^{n},C(K,E)),
\end{equation}
where this relation is interpreted as quasiasymptotics in the sense of \cite{P-S-V, Pil-Vin2014}. (The precise meaning of \eqref{wnweqpb1} is  $\|\langle\mathbf{G}(\varepsilon t),\rho(t)\rangle\|_{C(K,E)}=O(\varepsilon^{k}) $ for each $\rho\in \mathcal{D}(\mathbb{R}^{n})$.)
Now, it is well know (cf. \cite[Prop. 7.1, p. 22]{Pil-Vin2014} or \cite[Lemma 6]{zavialov88}) that the quasiasymptotic relation (\ref{wnweqpb1}) remains valid in the space $\mathcal{S}'(\mathbb{R}_{t}^{n},C(K,E))$. This fact means that we have the right to evaluate the relation (\ref{wnweqpb1}) at $\varphi\in\mathcal{S}(\mathbb{R}^{n})$, which immediately yields
\begin{align*}
\sup_{x\in K} \left\|M^{\mathbf{f}}_{\varphi}(x,y)\right\|&= \left\|\left\langle \mathbf{G}(yt),\varphi(t)\right\rangle\right\|_{C(K,E)}
\leq Cy^{k},
\end{align*}
for some $C>0$, as claimed.
\end{proof}

In the next sections we will focus on attention to regularizing transform with respect to non-degenerate test functions in the sense of Definition \ref{wnwd1}. It is clear that if $\int_{\mathbb{R}^{n}}\varphi(t)\mathrm{d}t\neq 0$, then $\varphi$ is non-degenerate. We conclude this section by discussing two integral transforms that arise as regularizing transforms with respect to test functions satisfying the latter condition.

\begin{example}[\emph{The regularizing transform as solution to Cauchy problems}]
\label{wnwex3.4}
 When the test function
is of certain special form, the regularizing transform can become the solution to a PDE. We discuss a particular case
in this example. Let the set $\Gamma\subseteq\mathbb{R}^{n}$ be a closed convex cone with vertex at the origin.
In particular, we may have $\Gamma=\mathbb{R}^{n}$. Let $P$ be a homogeneous polynomial of degree $d$ such that $\Re e\:P(iu)<0 \ \ \  \mbox{ for all } \;u\in\Gamma\setminus\left\{0\right\}.$ 
We denote \cite{vladimirovbook,vladimirov-d-z1} by $\mathcal{S}'_{\Gamma}\subseteq\mathcal{S}'(\mathbb{R}^{n})$ the subspace of distributions supported by $\Gamma$. Consider the $E$-valued Cauchy problem
\begin{equation}\label{wnweq3.5}
\frac{\partial}{\partial t}\mathbf{U}(x,t)=P\left(\frac{\partial}{\partial x}\right)\mathbf{U}(x,t), \ \ \ \lim_{t\to0^{+}}\mathbf{U}(x,t)=\mathbf{f}(x)\ \ \ \mbox{in }\mathcal{S}'(\mathbb{R}^{n}_{x}),
\end{equation}
$$\operatorname*{supp}\hat{\mathbf{f}}\subseteq\Gamma, \ \ \ (x,t)\in \mathbb{H}^{n+1},$$
within the class of $E$-functions of slow growth over $\mathbb{H}^{n+1}$, that is,
$$
\sup_{(x,t)\in \mathbb{H}^{n+1}} \left\|\mathbf{U}(x,t)\right\|\left(t+\frac{1}{t}\right)^{-k_1}\left(1+\left|x\right|\right)^{-k_{1}}<\infty, \ \ \ \mbox{for some }k_{1},k_{2}\in\mathbb{N}.
$$
One readily verifies that (\ref{wnweq3.5}) has a unique solution satisfying the latter slow growth condition. Indeed,
$$
\mathbf{U}(x,t)=\frac{1}{(2\pi)^{n}}\left\langle \hat{\mathbf{f}}(u), e^{ix\cdot u}e^{tP(iu)}\right\rangle=\frac{1}{(2\pi)^{n}}\left\langle \hat{\mathbf{f}}(u), e^{ix\cdot u}e^{P\left(it^{1/d}u\right)}\right\rangle
$$
is the sought solution. We can find \cite{vladimirov-d-z1} a test
function $\eta\in\mathcal{S}(\mathbb{R}^{n})$ with the property
$\eta(u)=e^{P(iu)},\; u\in\Gamma.$
Choosing $\varphi\in\mathcal{S}(\mathbb{R}^{n})$ such that $\hat{\varphi}=\eta$, we express $\mathbf{U}$ as a
\[
\mathbf{U}(x,t)=\left\langle \mathbf{f}(\xi),\frac{1}{t^{n/d}}\varphi\left(\frac{x-\xi}{t^{1/d}}\right)\right\rangle=M_{\varphi}^\mathbf{f}(x,y), \ \ \ \mbox{with } y=t^{1/d}.
\]
In particular, when $P(u)=|u|^{2}$, (\ref{wnweq3.5}) is the Cauchy problem for the heat equation and $\varphi(\xi)=(2\sqrt{\pi})^{-n}e^{-\xi^{2}/4}$.

\end{example}
\begin{example}[\emph{Laplace transforms as  regularizing transforms}]\label{wnwex3.5} Let $\Gamma$ be a
closed convex acute cone \cite{vladimirovbook,vladimirov-d-z1}
with vertex at the origin. Its conjugate cone is denoted by
$\Gamma^{\ast}$. The definition of an acute cone tells us that
$\Gamma^{\ast}$ has non-empty interior, set
$C_{\Gamma}=\operatorname*{int} \Gamma^{\ast}$ and
$T^{C_{\Gamma}}=\mathbb{R}^{n}+iC_{\Gamma}.$ We denote by
$\mathcal{S}'_{\Gamma}(E)$ the subspace of $E$-valued tempered
distributions supported by $\Gamma$. Given
$\mathbf{h}\in\mathcal{S}'_{\Gamma}(E)$, its Laplace
transform \cite{vladimirovbook} is
$$
\mathcal{L}\left\{\mathbf{h};z\right\}=\left\langle \mathbf{h}(u),e^{iz\cdot u}\right\rangle, \ \ \ z\in T^{C_{\Gamma}};
$$
it is a holomorphic $E$-valued function on the tube domain
$T^{C_{\Gamma}}$. Fix $\omega\in C_{\Gamma}$. We may write
$\mathcal{L}\left\{\mathbf{h};x+i\sigma\omega\right\}$,
$x\in\mathbb{R}^{n}$, $\sigma>0$, as a $\phi-$transform. In fact,
choose $\eta_{\omega}\in\mathcal{S}(\mathbb{R}^{n})$ such that
$\eta_{\omega}(u)=e^{- \omega\cdot u},$  $u\in\Gamma. $
Then, with
$\hat{\varphi}_{\omega}=\eta_{\omega}$ and
$\hat{\mathbf{f}}=(2\pi)^{n}\mathbf{h}$,
\[
\mathcal{L}\left\{\mathbf{h};x+i\sigma\omega\right\}=M_{\varphi_{\omega}}^{\mathbf{f}}(x,\sigma).
\]
Notice that this is a particular case of Example \ref{wnwex3.4} with $P_{\omega}(\xi)=i \omega\cdot\xi$.
\end{example}



\section{Tauberian class estimates}

\label{wnwce}

We now establish the Tauberian nature of the estimate \begin{equation}
\label{wnwceeq1} \left\|M^{\mathbf{f}}_{\varphi}(x,y)\right\|\leq
C \frac{(1+y)^{k}\left(1+\left|x\right|\right)^{l}}{y^{k}}, \qquad
(x,y)\in\mathbb{H}^{n+1}.
\end{equation}
We call
(\ref{wnwceeq1}) a  \emph{global class estimate}. We will prove
that if $\mathbf{f}$ takes values in a ``broad'' locally convex
space which contains the narrower Banach space $E$, and if
$\mathbf{f}$ satisfies (\ref{wnwceeq1}) for a non-degenerate test
function $\varphi$, then, there is a distribution $\mathbf{G}$
with values in the broad space such that
$\operatorname*{supp}\hat{\mathbf{G}}\subseteq\left\{0\right\}$
and $\mathbf{f}-\mathbf{G}\in\mathcal{S}'(\mathbb{R}^{n},E).$ In
case when the broad space is a normed one, $\mathbf{G}$ simply reduces
 to a polynomial. This will be done in Subsection \ref{gce}.

We shall also investigate in Subsection \ref{lce} the consequences
of (\ref{wnwceeq1}) when it is only assumed to hold for
$(x,y)\in\mathbb{R}^{n}\times(0,1]$, we call it then a \emph{local
class estimate}. In this case the situation is slightly different
and we obtain that
$\mathbf{f}-\mathbf{G}\in\mathcal{S}'(\mathbb{R}^{n},E),$ where
$\hat{\mathbf{G}}$ has compact support but its support may not be
any longer the origin. Nevertheless, the support of the Fourier transform of correction term remains controlled by the so-called index of non-degenerateness of the test function, introduced below. 

The correction terms, in both cases of global and local estimates, can be eliminated if one augments the hypotheses by involving a convolution average with respect to another suitable  test function.

Throughout this section, unless specified, $X$ is assumed to be a
 locally convex topological vector space such
that the Banach space $E\subset X$ and  the inclusion mapping $E\rightarrow X$ is linear and continuous. Observe that
the transform (\ref{wnwneq1}) makes sense for $X$-valued
distributions as well.  Furthermore, in order to gain generality, we will work with an integral version of the class estimate \eqref{wnwceeq1}.

\subsection{Global class estimates}
\label{gce} We begin with a full characterization of $\mathcal{S}'_{0}(\mathbb{R}^{n}, E)$ in terms of the wavelet transform.
\begin{proposition}
\label{wnwcep1} Let
$\mathbf{f}\in\mathcal{S}'_{0}(\mathbb{R}^{n},X)$ and let
$\psi\in\mathcal{S}_{0}(\mathbb{R}^{n})$ be a non-degenerate
wavelet. The following two conditions,
\begin{align}
\label{wnwceeq2} &\mathcal{W}_{\psi}\mathbf{f}(x,y)\in E, \mbox{
  for almost all value of } (x,y)\in\mathbb{H}^{n+1},
  \\
  & \nonumber
  \mbox{ and it is measurable as an } E\mbox{-valued function},
\end{align}
and there are constants $k,l\in\mathbb{N}$ such that
\begin{equation}
\label{wnwceeq4}
\int_{0}^{\infty}\int_{\mathbb{R}^{n}}\left(\frac{1}{y}+y\right)^{-k}\left(1+\left|x\right|\right)^{-l}
\left\|\mathcal{W}_{\psi}\mathbf{f}(x,y)\right\|\mathrm{d}x\mathrm{d}y<\infty
\end{equation}
are necessary and sufficient for
$\mathbf{f}\in\mathcal{S}'_{0}(\mathbb{R}^{n},E)$.
\end{proposition}
\begin{proof}
The necessity is clear (Proposition \ref{wnwp2}). We show the
sufficiency. Let $\eta$ be a reconstruction wavelet for $\psi$. We
apply the wavelet synthesis operator to the function
$\mathbf{K}(x,y)=\mathcal{W}_{\psi}\mathbf{f}(x,y)$, this is valid
because our assumptions (\ref{wnwceeq2}) and (\ref{wnwceeq4}) ensure
that $\mathbf{K}\in\mathcal{S}'(\mathbb{H}^{n+1},E)$. So, set
$\tilde{\mathbf{f}}:=\mathcal{M}_{\eta}\mathbf{K}\in\mathcal{S}'_{0}(\mathbb{R}^{n},E)
\subset\mathcal{S}'_{0}(\mathbb{R}^{n},X).$
 We must therefore
show $\tilde{\mathbf{f}}=\mathbf{f}$. Let
$\rho\in\mathcal{S}_{0}(\mathbb{R}^{n})$. We have, by definition (cf. \eqref{wnwneq10}),
(\ref{wnwneq6}), and (\ref{wnwneq8}),
\begin{equation*}
\langle
\tilde{\mathbf{f}},\rho\rangle=\frac{1}{c_{\psi,\eta}}\int_{0}^{\infty}\int_{\mathbb{R}^{n}}\left\langle
\mathbf{f}(t),
\frac{1}{y^{n}}\bar{\psi}\left(\frac{t-x}{y}\right)\mathcal{W}_{\bar{\eta}}\rho(x,y)\right\rangle\frac{\mathrm{d}x\mathrm{d}y}{y}
\end{equation*}
and
\begin{equation*}
\left\langle
\mathbf{f},\rho\right\rangle=\frac{1}{c_{\psi,\eta}}\left\langle
\mathbf{f},\mathcal{M}_{\bar{\psi}}\mathcal{W}_{\bar{\eta}}\rho\right\rangle
 =\frac{1}{c_{\psi,\eta}}\left\langle
\mathbf{f}(t),\int_{0}^{\infty}\int_{\mathbb{R}^{n}}
\frac{1}{y^{n}}\bar{\psi}\left(\frac{t-x}{y}\right)\mathcal{W}_{\bar{\eta}}
\rho(x,y)\right\rangle\frac{\mathrm{d}x\mathrm{d}y}{y}\:
.
\end{equation*}
Thus, with
$$\Phi(x,y;t)=\frac{1}{y^{n+1}}\bar{\psi}\left(\frac{t-x}{y}\right)\mathcal{W}_{\bar{\eta}}\rho(x,y),$$
our problem reduces to justify the interchange of the integrals
with the dual pairing in
\begin{equation}
\label{wnwceeq5}
\int_{0}^{\infty}\int_{\mathbb{R}^{n}}\left\langle
\mathbf{f}(t),\Phi(x,y;t) \right\rangle\mathrm{d}x\mathrm{d}y=
\left\langle
\mathbf{f}(t),\int_{0}^{\infty}\int_{\mathbb{R}^{n}}\Phi(x,y;t)\mathrm{d}x\mathrm{d}y\right\rangle
.
\end{equation}
To show (\ref{wnwceeq5}), we verify that
\begin{equation}
\label{wnwceeq6} 
\left\langle
\mathbf{w}^\ast,\int_{0}^{\infty}\int_{\mathbb{R}^{n}}\left\langle
\mathbf{f}(t),\Phi(x,y;t)
\right\rangle\mathrm{d}x\mathrm{d}y\right\rangle=
\left\langle \mathbf{w}^\ast,\left\langle
\mathbf{f}(t),\int_{0}^{\infty}\int_{\mathbb{R}^{n}}\Phi(x,y;t)\mathrm{d}x\mathrm{d}y\right\rangle
\right\rangle ,
\end{equation}
 for arbitrary $\mathbf{w}^\ast\in X'$ (here is
where the local convexity and the Hausdorff property of $X$ play a role). Since the integral
involved in the left hand side of \eqref{wnwceeq6} is a
Bochner integral in $E$ and the restriction of $\mathbf{w}^{\ast}$
to $E$ belongs to $E'$, we obtain at once the exchange formula
\begin{equation}
\label{wnwceeq7} 
\left\langle
\mathbf{w}^\ast,\int_{0}^{\infty}\int_{\mathbb{R}^{n}}\left\langle
\mathbf{f}(t),\Phi(x,y;t)
\right\rangle\mathrm{d}x\mathrm{d}y\right\rangle=
\int_{0}^{\infty}\int_{\mathbb{R}^{n}}\left\langle
\mathbf{w}^\ast,\left\langle \mathbf{f}(t),\Phi(x,y;t)
\right\rangle\right\rangle\mathrm{d}x\mathrm{d}y.
\end{equation}
On the other
hand, we may write $
\int_{0}^{\infty}\int_{\mathbb{R}^{n}}\Phi(x,y;t)\mathrm{d}x\mathrm{d}y
$
as the limit of Riemann sums, convergent in
$\mathcal{S}_{0}(\mathbb{R}^{n}_{t})$, we then easily justify the
exchanges that yield
\begin{equation}
\label{wnwceeq8}
\left\langle \mathbf{w}^\ast,\left\langle
\mathbf{f}(t),\int_{0}^{\infty}\int_{\mathbb{R}^{n}}\Phi(x,y;t)\mathrm{d}x\mathrm{d}y\right\rangle
\right\rangle=
\int_{0}^{\infty}\int_{\mathbb{R}^{n}}\left\langle
\mathbf{w}^\ast,\left\langle \mathbf{f}(t),\Phi(x,y;t)
\right\rangle\right\rangle\mathrm{d}x\mathrm{d}y.
\end{equation}
The equality (\ref{wnwceeq6}) follows now by comparing
(\ref{wnwceeq7}) and (\ref{wnwceeq8}).
\end{proof}

We now consider the general case of regularizing transforms with respect to non-degenerate kernels (cf. Definition \ref{wnwd1}).
\begin{theorem}
\label{wnwceth1} Let $\mathbf{f}\in\mathcal{S}'(\mathbb{R}^{n},X)$
and let $\varphi\in\mathcal{S}(\mathbb{R}^{n})$ be non-degenerate. Sufficient conditions for the
existence of an $X$-valued distribution $\mathbf{G}\in
\mathcal{S}'(\mathbb{R}^{n},X)$ such that
$\mathbf{f}-\mathbf{G}\in\mathcal{S}'(\mathbb{R}^{n},E)$ and
$\operatorname*{supp} \hat{\mathbf{G}}\subseteq\left\{0\right\}$
are:
\begin{enumerate}
\item[(i)] $M_{\varphi}^{\mathbf{f}}(x,y)$ takes values in $E$ for almost
all $(x,y)\in\mathbb{H}^{n+1}$ and is measurable as an $E$-valued function.
\item[(ii)] There exist $k,l\in\mathbb{N}$ such that 
\begin{equation}
\label{global ce eq}
\int_{0}^{\infty}\int_{\mathbb{R}^{n}}\left(\frac{1}{y}+y\right)^{-k}\left(1+\left|x\right|\right)^{-l}
\left\|M_{\varphi}^{\mathbf{f}}(x,y)\right\|\mathrm{d}x\mathrm{d}y<\infty
\end{equation}
\end{enumerate}
In this case, we also have
\begin{equation}
\label{wnwceeq10} P_{q}^{\varphi}\left(\frac{\partial}{\partial
t}\right)\mathbf{f}\in \mathcal{S}'(\mathbb{R}^{n},E) , \qquad
\mbox{for all } q\in\mathbb{N},
\end{equation}
where $P_{q}^{\varphi}$ are the homogeneous terms of the Taylor
polynomials of $\hat{\varphi}$ at the origin (cf. \eqref{taylorhomoeq})
\end{theorem}
\begin{proof} 
Consider the non-degenerate wavelet $\psi_{1}\in\mathcal{S}_{0}(\mathbb{R}^{n})$ given by $\hat{\psi}_{1}(u)=e^{-\left|u\right|-(1/\left|u\right|)}$. Set $\psi=\check{\bar{\varphi}}\ast\psi_{1}$; then, $\psi\in\mathcal{S}_{0}(\mathbb{R}^{n})$ is also a non-degenerate wavelet. Using the same argument as in the proof of
Proposition \ref{wnwcep1}, the exchange of integral and dual
paring performed in the following calculation is valid,
\begin{align*}
\mathcal{W}_{\psi}\mathbf{f}(x,y)&=\left\langle \mathbf{f}(x+yt),(\check{\varphi}\ast\bar{\psi}_{1})(t)\right\rangle
\\
&
=\left\langle \mathbf{f}(x+yt),\int_{\mathbb{R}^{n}}\bar{\psi}_{1}(u)\varphi(u-t)\mathrm{d}u \right\rangle
\\
&=\int_{\mathbb{R}^{n}}\bar{\psi}_{1}(u) \left\langle \mathbf{f}(x+yt),\varphi(u-t) \right\rangle\mathrm{d}u
\\
&
=\int_{\mathbb{R}^{n}}M^{\mathbf{f}}_{\varphi}(x+yu,y)\bar{\psi}_{1}(u)\mathrm{d}u.
\end{align*}
where the integral is taken in the sense of Bochner. Thus, the
restriction of $\mathbf{f}$ to $\mathcal{S}_{0}(\mathbb{R}^{n})$
is readily seen to satisfy the hypotheses of Proposition \ref{wnwcep1}, and hence
there exists $\mathbf{g}\in\mathcal{S}'(\mathbb{R}^{n},E)$ such
that $\left\langle \mathbf{f-g},\rho\right\rangle=0$ for all
$\rho\in\mathcal{S}_{0}(\mathbb{R}^{n})$. This gives at once that
$\mathbf{G}=\mathbf{f-g}$ satisfies $\operatorname*{supp}
\hat{\mathbf{G}}\subseteq\left\{0\right\}$ and
$\mathbf{f-G}\in\mathcal{S}'(\mathbb{R}^{n},E)$. 

Next, it is clear that $\mathbf{G}$ is an $X$-valued entire function with power series expansion, say, 
$
\mathbf{G}(t)=\sum_{m\in\mathbb{N}^{n}}i^{\left|m\right|} t^{m} \mathbf{w}_{m},$ with $\mathbf{w}_{m}\in X,
$
so that $\hat{\mathbf{G}}$ is given by the multipole series
$$
\hat{\mathbf{G}}(u)=\sum_{m\in\mathbb{N}^{n}} (-1)^{|m|} \delta^{(m)}(u)\mathbf{w}_{m},
$$
convergent in $\mathcal{S}'(\mathbb{R}^{n},X)$. The relation (\ref{wnwceeq10}) would follow
immediately if we show that
$P_{q}^{\varphi}\left(\partial/\partial t\right)\mathbf{G}$ is 
$E$-valued. Observe that the hypotheses imply that
$M_{\varphi}^{\mathbf{G}}(x,y)\in E$, for almost all $(x,y)$.
Hence, for almost all $(x,y)$,
\begin{align*}
M_{\varphi}^{\mathbf{G}}(x,y)&
=\frac{1}{(2\pi)^{n}}\left\langle \hat{\mathbf{G}}(u),e^{ix\cdot
u}{\hat{\varphi}}(yu) \right\rangle
=
\sum_{m\in \mathbb{N}}
\frac{\partial^{\left|m\right|}}{\partial u^{m}}\left.\left(
e^{ix\cdot u}{\hat{\varphi}}(yu)\right)\right|_{u=0}
\mathbf{w}_{m}
\\
& =
\sum_{q=0}^{\infty}y^{q}\sum_{\left|j\right|=q}{\hat{\varphi}^{(j)}(0)}\sum_{j\leq m}\binom{m}{j}(ix)^{m-j}\mathbf{w}_{m}
\\
& = \sum_{q=0}^{\infty}(iy)^{q}(P^{\varphi}_{q}\left(\partial/\partial
x\right)\mathbf{G})(x)\in E.
\end{align*}
But the latter readily implies that
$(P_{q}^{\varphi}\left(\partial/\partial
x\right)\mathbf{G})(x)\in E,$ for all $q\in\mathbb{N}$ and 
$x\in\mathbb{R}^{n}.$

\end{proof}

When $X$ is a normed space, we obviously have that the only
$X$-valued distributions with support at the origin are precisely
those having the form 
$$
\sum_{\left|m\right|\leq
N}\delta^{(m)}\mathbf{w}_{m}, \ \ \ \mathbf{w}_{m}\in X.
$$
Thus, we
have,

\begin{corollary}
\label{wnwcec1} Let $X$ be a normed space, $\mathbf{f}\in\mathcal{S}'(\mathbb{R}^{n},X)$,
and let $\varphi\in\mathcal{S}(\mathbb{R}^{n})$ be non-degenerate. Then, the conditions
(i) and (ii) of Theorem \ref{wnwceth1} imply the existence of an
$X$-valued polynomial $\mathbf{P}$ such that
$\mathbf{f-P}\in\mathcal{S}'(\mathbb{R}^{n},E)$. Furthermore,
$$\mathbf{f}\in \mathcal{S}_{I_{\varphi}}(\mathbb{R}^{n},E),$$ 
where $I_{\varphi}$ is the ideal generated by the homogeneous terms of the Taylor polynomial of $\hat{\varphi}$ at the origin (cf. \eqref{taylorhomoeq} and \eqref{classestIdealeq}).
\end{corollary}
\begin{proof} It remains to verify that $\mathbf{f}\in \mathcal{S}_{I_{\varphi}}(\mathbb{R}^{n},E)$. 
By Theorem \ref{wnwceth1}, we have that if  $\mathbf{P}$ is an $X$-valued polynomial such that
$\mathbf{f}-\mathbf{P}\in\mathcal{S}'(\mathbb{R}^{n},E)$, then  the polynomials $P_{q}^{\varphi}\left(\frac{\partial}{\partial
t}\right)\mathbf{P}$ are $E$-valued for all $q\in\mathbb{N}.$ We must show that if $\phi \in \mathcal{S}_{I_{\varphi}}(\mathbb{R}^{n})$, then $\int_{\mathbb{R}^{n}}\phi(t)\mathbf{P}(t)\mathrm{d}t\in E$. Let $N$ be the degree of $\mathbf{P}$. There are polynomials $Q_{q}$ such that 

$$T_{N}(u)=\sum_{|m|\leq N} \frac{\hat{\phi}^{(m)}(0)}{m!} u^{m}= \sum_{q=0}^{N} Q_{q}(u) P^{\varphi}_{q}(u).$$
So,
\begin{align*}
 \int_{\mathbb{R}^{n}}\phi(t)\mathbf{P}(t)\mathrm{d}t&=\langle \mathbf{P}\left(-i\partial/\partial u\right)\delta (u), \hat{\phi}(u)\rangle 
 =\sum_{q=0}^{N}\langle \mathbf{P}\left(-i\partial/\partial u\right)\delta (u), Q_q(u)P^{\varphi}_{q}(u)\rangle 
 \\
 &
 =
 \sum_{q=0}^{N} (-i)^{q} Q_{q}(-i\partial/\partial u) (P^{\varphi}_{q}(\partial/\partial u)\mathbf{P}) (0)\in E.
\end{align*}
\end{proof}

Note that the degree of the polynomial $\mathbf{P}$ occurring in Corollary \ref{wnwcec2} may depend merely on $\mathbf{f}$, and not on the test function. However, when the Taylor polynomials of $\hat{\varphi}$ posses a rich algebraic structure, it is possible to say more about the degree of $\mathbf{P}$. Recall $\mathbb{P}_{d}(\mathbb{R}^{n})$ was defined in Subsection \ref{wnw sub spaces}.

\begin{corollary}
\label{wnwcec2} Let the hypotheses of Corollary \ref{wnwcec1} be satisfied. If there exists $d\in\mathbb{N}$ such that $\mathbb{P}_{d}(\mathbb{R}^{n})$
is contained in the ideal generated by the polynomials $P_{0}^{\varphi}, \dots, P_{d}^{\varphi}$, then there exists an $X$-valued polynomial $\mathbf{P}$ of degree at most $d-1$ such that $\mathbf{f}-\mathbf{P}\in\mathcal{S}'(\mathbb{R}^{n},E)$. In particular, $\mathbf{f}\in\mathcal{S}'_{\mathbb{P}_{d}}(\mathbb{R}^{n},E)$.
\end{corollary}
\begin{proof}
Corollary \ref{wnwcec1} yields the existence of an $X$-valued
polynomial $\tilde{\mathbf{P}}(t)=\mathbf{P}(t)+
\sum_{d\leq\left|m\right|\leq N}\mathbf{w}_{m}t^{m}$ such that
$\mathbf{f}-\tilde{\mathbf{P}}\in \mathcal{S}'(\mathbb{R}^{n},E)$
and $\mathbf{P}$ has degree at most $d-1$. Then, we must show that
$\mathbf{w}_{m}\in E$ for $d\leq\left|m\right|\leq N.$ But
Corollary \ref{wnwcec1} also implies that
$Q(\partial/\partial t)\tilde{\mathbf{P}}$ is an
$E$-valued polynomial for any $
Q\in I_{\varphi}$, and since
$\mathbb{P}_{d}(\mathbb{R}^{n})\subseteq I_{\varphi}$, we obtain
at once that $\mathbf{w}_{m}=m!((\partial^{|m|}/\partial
t^{m})\tilde{\mathbf{P}})(0)\in E,$ for 
$d\leq\left|m\right|\leq N.$
\end{proof}

In general, it is not possible to replace the $X$-valued entire function $\mathbf{G}$ by an
$X$-valued polynomial in Theorem \ref{wnwceth1}. However, we know
some valuable information about $\hat{\mathbf{G}}$. Since it is
supported by the origin, we have already observed that
$$\hat{\mathbf{G}}=\sum_{m\in\mathbb{N}^{n}}\frac{(-1)^{\left|m\right|}\delta^{(m)}}{m!}\mu_{m}(\hat{\mathbf{G}})
,$$ where the vectors $\mu_{m}(\hat{\mathbf{G}})=\left\langle
\hat{\mathbf{G}}(u),u^{m}\right\rangle\in X
$
are actually its moments
and the series is convergent in $\mathcal{S}'(\mathbb{R}^{n},X)$.
This series is ``weakly finite'', in the sense that for each
$\mathbf{w}^{\ast}\in X'$ there exists
$N_{\mathbf{w}^{\ast}}\in\mathbb{N}$ such that
$$
\langle \mathbf{w}^{\ast},\langle
\hat{\mathbf{G}},\rho\rangle\rangle=\sum_{\left|m\right|\leq
N_{\mathbf{w}^{\ast}}}\frac{\rho^{(m)}(0)}{m!}\left\langle
\mathbf{w}^{\ast}, \mu_{m}(\hat{\mathbf{G}})\right\rangle, \ \ \
\mbox{for all } \rho\in\mathcal{S}(\mathbb{R}^{n}).
$$
Furthermore, given any continuous seminorm $\mathfrak{p}$ on $X$,
one can find an $N_{\mathfrak{p}}\in\mathbb{N}$ such that
$$
\mathfrak{p}\left(\langle\hat{\mathbf{G}},\rho\rangle-\sum_{\left|m\right|\leq
N_{\mathfrak{p}}}\frac{\rho^{(m)}(0)}{m!}\mu_{m}(\hat{\mathbf{G}})\right)=0,
$$
for all $\rho\in\mathcal{S}(\mathbb{R}^{n})$. Finally, as we have also already mentioned, its inverse Fourier transform $\mathbf{G}$ can be naturally
identified with an entire $X$-valued function.
\begin{example} \label{wnwceex1} We consider $X=C(\mathbb{R})$ and $E=C_{b}(\mathbb{R})$,
 the space of continuous bounded functions. Let
 $\chi_{q}\in C(\mathbb{R})$ be non-trivial such that
 $\operatorname*{supp }\chi_{q}\subset(q,q+1)$, $q\in\mathbb{N}$. Furthermore, for each $q\in\mathbb{N}$ find a harmonic homogeneous polynomial $Q_{q}$ of degree $q$, i.e., $\Delta Q_{q}=0$. Consider the $E$-valued distribution
$$
\mathbf{G}(t,\xi)=\sum_{q=0}^{\infty}Q_{q}(t)\chi_{\nu}(\xi)\in
\mathcal{S}'(\mathbb{R}^n_{t}, C(\mathbb{R}_{\xi}))\setminus
\mathcal{S}'(\mathbb{R}^n_{t}, C_{b}(\mathbb{R}_{\xi})) .
$$
Its Fourier transform is given by an infinite multipole series
supported at the origin, i.e.,
$$
\hat{\mathbf{G}}(u,\xi)=(2\pi)^n\sum_{q=0}^{\infty}
\left(Q_{q}\left(i\partial/\partial u\right)\delta\right)(u)\:\chi_{\nu}(\xi).
$$
Let $\mathbf{h}\in \mathcal{S}'(\mathbb{R}^{n}, C_{b}(\mathbb{R}))$ and
let $\varphi\in\mathcal{S}(\mathbb{R}^{n})\setminus\mathcal{S}_{0}(\mathbb{R}^{n})$ be a non-degenerate test function such that its Fourier transform satisfies $\hat{\varphi}(u)=\left|u\right|^2+O(\left|u\right|^N)$ as $u\to0$, for all $N>2$. If 
$
\mathbf{f}=\mathbf{h}+\mathbf{G}\in
\mathcal{S}'(\mathbb{R}^n, C(\mathbb{R})),
$
then $M_{\varphi}^{\mathbf{f}}(x,y)=M_{\varphi}^{\mathbf{h}}(x,y)$ for all $(x,y)\in\mathbb{H}^{n+1}$. Thus,
$\mathbf{f}$ satisfies all the hypotheses of Theorem
\ref{wnwceth1}; however, there is no $C(\mathbb{R})$-valued
polynomial $\mathbf{P}$ such that $\mathbf{f-P}\in
\mathcal{S}'(\mathbb{R}^{n}, C_{b}(\mathbb{R}))$.
\end{example}

\bigskip

The occurrence of the correction term $\mathbf{G}$ in Theorem \ref{wnwceth1} can be eliminated if one augments the hypotheses by involving a convolution average of $\mathbf{f}$ with respect to another test function as follows. One then obtains a characterization of $\mathcal{S}'(\mathbb{R}^{n}, E)$.

\begin{theorem}
\label{wnwceth2} Let $\mathbf{f}\in\mathcal{S}'(\mathbb{R}^{n},X)$, let $\varphi\in\mathcal{S}(\mathbb{R}^{n})$ be non-degenerate, and let $\varphi_{0}\in\mathcal{S}(\mathbb{R}^{n})$ be such that $\int_{\mathbb{R}^{n}} \varphi_{0}(t)\mathrm{d}t\neq 0$. Then, 
$\mathbf{f}\in\mathcal{S}'(\mathbb{R}^{n},E)$ if and only if the following three conditions hold:
\begin{enumerate}
\item[(i)] $M_{\varphi}^{\mathbf{f}}(x,y)$ takes values in $E$ for almost
all $(x,y)\in\mathbb{H}^{n+1}$ and is measurable as an $E$-valued function.
\item[(ii)] There exist $k,l\in\mathbb{N}$  such that \eqref{global ce eq} holds.
\item [(iii)] There is $l\in\mathbb{N}$ such that $(\mathbf{f}\ast \varphi_{0})(x)\in E$ a.e., it is measurable, and satisfies
\begin{equation}
\label{eq1exTce}
\int_{\mathbb{R}^{n}}(1+|x|)^{-l}
\|(\mathbf{f}\ast \varphi_{0})(x)\|<\infty.\end{equation}
\end{enumerate}
\end{theorem}

\begin{proof} We apply Theorem \ref{wnwceth1} to obtain $\mathbf{G}\in \mathcal{S}'(\mathbb{R}^{n}, X)$ such that $\mathbf{f}-\mathbf{G}\in \mathcal{S}'(\mathbb{R}^{n}, E)$ and $\operatorname*{supp}\hat{\mathbf{G}}\subset \{0\}$. Using (iii), we conclude that, a.e., $(\mathbf{G}\ast \varphi_{0})(x)\in E$ and that this function defines an element of $\mathcal{S}'(\mathbb{R}^{n},E)$ as a tempered $E$-valued function. Let $\sigma>0$ be sufficiently small such that $|\hat{\varphi}(u)|>0$ for $|u|<\sigma$. Since $\operatorname*{supp}\hat{\mathbf{G}}\subset \{0\}$, it suffices to show that $\langle \mathbf{G}, \hat{\rho} \rangle\in E $ for each $\rho \in \mathcal{D}(B(0,\sigma))$, whence we would have $\mathbf{G}\in \mathcal{S}'(\mathbb{R}^{n}, E)$. Now, setting $\hat{\chi}= \rho/\hat{\varphi}_{0}$,
\begin{align*}
\langle \mathbf{G}, \hat{\rho} \rangle&=\langle \hat{\mathbf{G}}, \hat{\chi} \cdot \hat{\varphi}_{0} \rangle
\\
&
=(2\pi)^{n}\left\langle \mathbf{G}(t),
\int_{\mathbb{R}^{n}}\chi(-\xi)\varphi_{0}\left(\xi-t\right)
\mathrm{d}\xi\right\rangle 
\\
&
=(2\pi)^{n}\int_{\mathbb{R}^{n}}\chi(-\xi) (\mathbf{G}\ast \varphi_{0})(\xi)\mathrm{d}\xi \in E,
\end{align*}
where the exchange with the integral sign can be established as in
the proof of Proposition \ref{wnwcep1} and the very last integral is taken in the Bochner sense. 
\end{proof}

It should be noticed that if $\varphi$ is such that $\int_{\mathbb{R}^{n}} \varphi(t)\mathrm{d}t\neq 0$, then one may use $\varphi=\varphi_{0}$ in Theorem \ref{wnwceth2} and the condition (iii) becomes part of (ii); a stronger result is however stated below in Corollary \ref{wnwcephi}.

\subsection{Local class estimates}
\label{lce}

We now proceed to study local class estimates, namely,
(\ref{wnwceeq1}) only assumed to hold for
$(x,y)\in\mathbb{R}^{n}\times(0,1]$. We work again with an integral condition instead of a pointwise bound in order to gain generality. Let us start by pointing out
that $M^{\mathbf{f}}_{\varphi}(x,y)$ may sometimes be trivial for
$y\in(0,1)$, this may happen even if $\varphi$ is non-degenerate:

\begin{example}
\label{wnwceex2} Let $\omega\in\mathbb{S}^{n-1}$,
$r\in\mathbb{R}_{+}$; denote $[0,r\omega]=\left\{\sigma\omega:\:
\sigma\in[0,r]\right\}$. Suppose that $\mathbf{f}\in
\mathcal{S}'(\mathbb{R}^{n},X)$ is such that
$\operatorname*{supp}\:\hat{\mathbf{f}}\subset [0,r\omega]$ and
$\varphi\in \mathcal{S}(\mathbb{R}^{n})$ is any test function
satisfying
$\operatorname*{supp}\hat{\varphi}\subset\mathbb{R}^{n}\setminus
[0,r\omega]$, then
\begin{equation*}
M_{\varphi}^\mathbf{f}(x,y)=\frac{1}{(2\pi)^{n}}\left\langle
\hat{\mathbf{f}}(u),e^{ixu}\hat{\varphi}(yu)\right\rangle=0,
\ \ \ \mbox{for all } y\in(0,1).
\end{equation*}
\smallskip
\end{example}

Fortunately, we will show that the only distributions
$\mathbf{f}\in\mathcal{S}'(\mathbb{R}^{n},X)$
that may satisfy a local class estimate, with respect to a
non-degenerate test function, are, modulo elements of $\mathcal{S}'(\mathbb{R}^{n},E)$, those whose Fourier transforms are
compactly supported.

We need to introduce some terminology in order to move further on. We
will make use of weak integrals (Pettis integrals) for $X$-valued functions as
defined, for example, in \cite[p. 77]{rudin}. We say that a
tempered $X$-valued distribution
$\mathbf{g}\in\mathcal{S}'(\mathbb{R}^{n},X)$ is \emph{weakly
regular} if there exists an $X$-valued function
$\tilde{\mathbf{g}}$ such that $\rho \tilde{\mathbf{g}}$ is weakly
integrable over $\mathbb{R}^{n}$ for all
$\rho\in\mathcal{S}(\mathbb{R}^{n})$ and
$$
\left\langle
\mathbf{g},\rho\right\rangle=\int_{\mathbb{R}^{n}}\rho(t)\tilde{\mathbf{g}}(t)\mathrm{d}t
\in X ,
$$
where the last integral is taken in the weak sense. We identify
$\mathbf{g}$ with $\tilde{\mathbf{g}}$, so, as usual, we write
$\mathbf{g}=\tilde{\mathbf{g}}$. 

Let us recall some facts about (vector-valued) compactly supported
distributions. Let
$\mathbf{g}\in\mathcal{S}'(\mathbb{R}^{n},X)$ have support in
$\overline{B(0,r)}$, the closed ball of radius $r$. Then, the
following version of the Schwartz-Paley-Wiener theorem holds:
$$\mathbf{G}(z)=\left\langle \mathbf{g}(u),e^{-iz\cdot
u}\right\rangle, \ \ \ z\in \mathbb C^n,$$ is an $X$-valued entire
function which defines a weakly regular tempered distribution, and
$\mathbf{G}(\xi)=\hat{\mathbf{g}}(\xi),$ $\xi  \in \mathbb R^n$;
moreover, $\mathbf{G}$ is of weakly exponential type, i.e., for
all $\mathbf{w}^{\ast}\in X'$ one can find constants
$C_{\mathbf{w}^{\ast}}>0$ and $N_{\mathbf{w}^{\ast}}\in\mathbb{N}$
with
\begin{equation}
\label{wnwceeq11} \left|\left\langle
\mathbf{w}^{\ast},\mathbf{G}(z)\right\rangle\right|\leq
C_{\mathbf{w}^{\ast}}(1+\left|z\right|)^{N_{\mathbf{w}^{\ast}}}e^{r\left|\Im
m\:z\right|}, \ \ \ z\in\mathbb{C}^{n}.
\end{equation}
Conversely, if $\mathbf{G}$ is an $X$-valued entire function which
defines a weakly regular  tempered distribution and for all
$w^*\in X'$ there exist $C_{w^*}>0$ and $ N_{w^*}\in\mathbb N$ such
that (\ref{wnwceeq11}) holds, then
$\hat{\mathbf{G}}=\mathbf{g}
,$ where
${\mathbf{g}}\in\mathcal{S}'(\mathbb R^n,X)$ and
$\operatorname*{supp}{\mathbf{g}}\subseteq \overline{B(0,r)}$.

The following concept for non-degenerate test functions is of much
relevance for the problem under consideration.

\begin{definition}
\label{wnwced1} Let $\varphi\in\mathcal{S}(\mathbb{R}^{n})$ be
non-degenerate. Given $\omega\in\mathbb{S}^{n-1}$, consider the
function of one variable $R_{\omega}(r)=\hat{\varphi}(r\omega)\in
C^{\infty}[0,\infty)$. We define the index of non-degenerateness
of $\varphi$ as the (finite) number
$$
\tau=\inf\left\{r\in\mathbb{R}_{+}:\:\operatorname*{supp}
R_{\omega}\cap[0,r]\neq \emptyset,
\forall\omega\in\mathbb{S}^{n-1}\right\} .
$$
\end{definition}

\bigskip

We are ready to state and prove the main Tauberian result of this
subsection. We shall consider slightly more general norm
estimates for the regularizing transform
$M^{\mathbf{f}}_{\varphi}$ in terms of functions
$\Psi:\mathbb{R}^{n}\times(0,1]\to\mathbb{R}_{+}$ which satisfy,
for some constants $C_1,C_2>0$ and $k,l\in\mathbb{N}$,
\begin{equation}
\label{wnweqgb}
\Psi(0,y)\geq C_1y^{k}\ \ \mbox{ and }  \ \ \Psi(x+\xi,y)\leq C_2 \Psi(x,y)(1+\left|\xi\right|)^{l},
\end{equation}
for all $x,\xi\in\mathbb{R}^{n}$ and $y\in(0,1]$. In the next theorem $L^{p,p'}_{\Psi}((0,1]\times\mathbb{R}^{n},E)$ stands for the mixed $L^{p,p'}$-space of $E$-valued functions on $(0,1]\times\mathbb{R}^{n}$ with respect to the weight $\Psi$, that is, the space of $E$-valued measurable functions $\mathbf{F}$ such that 
$$
 \int_{0}^{1}\left(\int_{\mathbb{R}^{n}}\left(\|\mathbf{F}(x,y)\|\Psi(x,y)\right)^{p}\mathrm{d}x\right)^{\frac{p'}{p}} \frac{\mathrm{d}y}{y}<\infty
$$
The parameters are assumed to satisfy $p,p'\in [1,\infty]$. 
\begin{theorem}
\label{wnwceth3} Let $\mathbf{f}\in\mathcal{S}'(\mathbb{R}^{n},X)$
and let $\varphi\in\mathcal{S}(\mathbb{R}^{n})$ be a
non-degenerate test function with index of non-degenerateness
$\tau$. Assume:
\begin{enumerate}
\item[(i)] $M_{\varphi}^\mathbf{f}(x,y)$ takes values in $E$ for almost all $(x,y)\in\mathbb{R}^{n}\times(0,1]$ and is measurable as an $E$-valued function on $\mathbb{R}^{n}\times(0,1]$.
\item[(ii)] There is a function $\Psi:\mathbb{R}^{n}\times(0,1]$ that satisfies \eqref{wnweqgb} and such that $M_{\varphi}^{\mathbf{f}}\in L^{p,p'}_{\Psi}((0,1]\times \mathbb{R}^{n}, E).$
\end{enumerate}
Then, for any $r>\tau$, there exists an $X$-valued entire function
$\mathbf{G}$, which defines a weakly regular tempered $X$-valued distribution
and satisfies \eqref{wnwceeq11}, such that
$$\mathbf{f-G}\in\mathcal{S}'(\mathbb{R}^{n},E).$$
Furthermore, $M^{\mathbf{f-G}}_{\varphi}\in L^{p,p'}_{\Psi}((0,1]\times \mathbb{R}^{n}, E)
$
and we can choose $\mathbf{G}$ so that $\hat{\mathbf{G}}=\chi\hat{\mathbf{f}}$, where $\chi\in\mathcal{D}(\mathbb{R}^{n})$ is an arbitrary test function that satisfies $\chi(t)=1$ for $\left|t\right|\leq \tau$ and has support contained in the ball of radius $r$ and center at the origin.
\end{theorem}
\begin{proof}
Let $r_{1}$ be such that $\tau<r_{1}<r$. It is easy to find a reconstruction wavelet $\eta\in\mathcal{S}_{0}(\mathbb{R}^{n})$ for $\varphi$,
in the sense that
$$
1=\int^{\infty}_{0}\hat{\varphi}(r\omega)\hat{\eta}(r\omega)\frac{\mathrm{d}r}{r}, \ \ \ \mbox{for every }\omega\in\mathbb{S}^{n-1},
$$
with the
property $\operatorname*{supp}\hat{\eta}\subset B(0,r_{1})$. Indeed, if we choose a non-negative $\kappa\in\mathcal{D}(\mathbb{R}^{n})$ with support in $B(0,r_{1})\setminus\left\{0\right\}$ and being equal to $1$ in a neighborhood of the sphere $\tau \mathbb{S}^{n-1}=\left\{u\in\mathbb{R}^{n}: \left|u\right|=\tau\right\}$, then 
the same argument given in the proof of \cite[Prop. 5.1, p. 11]{Pil-Vin2014} shows that 
$$
\hat{\eta}(x)= \frac{\kappa(x)\overline{\hat{\varphi}(x)}}
{\displaystyle\int^{\infty}_{0}\kappa(rx)\left|\hat{\varphi} (rx)\right| ^{2}\frac{\mathrm{d}r}{r}
}
$$
fulfills the requirements. The usual calculation \cite[p. 66]{holschneider} is valid and so, for $\rho\in\mathcal{S}_{0}(\mathbb{R}^{n})$,

$$\rho(t)=\int_{0}^{\infty}\int_{\mathbb{R}^{n}}
\frac{1}{y^{n}}\varphi\left(\frac{x-t}{y}\right)\mathcal{W}_{\bar{\eta}}
\rho(x,y)\frac{\mathrm{d}x\mathrm{d}y}{y}\:.$$
Observe now that if
$\operatorname*{supp}\hat{\rho}\subseteq\mathbb{R}^{n}\setminus
B(0,r_{1})$, then
$$
\mathcal{W}_{\bar{\eta}}\rho(x,y)=\frac{1}{(2\pi)^{n}}\int_{\mathbb{R}^{n}}e^{ix
\cdot u}\hat{\rho}(u)\hat{\eta}(-yu)\mathrm{d}u=0 , \ \ \
\mbox{for all } y\in[1,\infty).
$$
Thus, the same argument employed in Proposition \ref{wnwcep1}
applies to show
\begin{equation}
\label{wnwceeq12}
\left\langle f,\rho \right\rangle=
\int_{0}^{1}\int_{\mathbb{R}^{n}}M_{\varphi}^\mathbf{f}(x,y)\mathcal{W}_{\bar{\eta}}\rho(x,y)\:\frac{\mathrm{d}x\mathrm{d}y}{y}\:
,
\end{equation}
for all $\rho\in\mathcal{S}(\mathbb{R}^{n})$ with
$\operatorname*{supp}\hat{\rho}\subseteq\mathbb{R}^{n}\setminus
B(0,r_{1}).$ Choose $\chi_{1}\in C^{\infty}(\mathbb{R}^{n})$ such that
$\chi_1(u)=1$ for all $u\in\mathbb{R}^{n}\setminus B(0,r)$ and
$\operatorname*{supp}\chi_1\in\mathbb{R}^{n}\setminus B(0,r_{1})$.
Now, $\hat{\chi}_1\ast\mathbf{f}$ is well defined since
$\hat{\chi}_1\in \mathcal{O}'_{C}(\mathbb{R}^{n})$ (the space of
convolutors), and actually (\ref{wnwceeq12}) and the continuity of $\mathcal{W}_{\bar{\eta}}$ imply that
$(2\pi)^{-n}\hat{\chi}_1\ast\mathbf{f}\in\mathcal{S}'(\mathbb{R}^{n},E)$.
Therefore,
$\mathbf{G}=\mathbf{f}-(2\pi)^{-n}\hat{\chi}_1\ast\mathbf{f}$
satisfies the requirements because $\hat{\mathbf{G}}=\chi\mathbf{\hat{f}}$, where $\chi(\xi)=1-\chi_1(-\xi)$, and so
$\operatorname*{supp}\hat{\mathbf{G}}\subseteq \overline{B(0,r)}$. Since $\hat{\chi}_{1}(\xi)=(2\pi)^{n}\delta(\xi) -\hat{\chi}(-\xi)$ and so
\begin{align*}
M^{\mathbf{f-G}}_{\varphi}(x,y)&=M^{\mathbf{f}}_{\varphi}(x,y)-\frac{1}{(2\pi)^{n}}\left\langle\mathbf{f}(\xi),\left\langle \frac{1}{y^{n}}\varphi\left(\frac{x+t-\xi}{y}\right),\hat{\chi}(t)\right\rangle \right\rangle
\\
&
=M^{\mathbf{f}}_{\varphi}(x,y)-\frac{1}{(2\pi)^{n}}\int_{\mathbb{R}^{n}}M^{\mathbf{f}}_{\varphi}(x+t,y)\hat{\chi}(t)\mathrm{d}t,
\end{align*}
we readily obtain the $L^{p,p'}$-norm estimate for
 $M^{\mathbf{f-G}}_{\varphi}$.
  \end{proof}


One may be tempted to think that in Theorem \ref{wnwceth3} it is possible to take $\mathbf{G}$ with support in $\overline{B(0,\tau)}$; however, this is not true, in general, as the following counterexample shows.
\begin{example}
\label{wnwceex3} Let $X$, $E$, and the sequence
$\left\{\chi_{q}\right\}_{q=1}^{\infty}$ be as in Example
\ref{wnwceex1}. We work in dimension $n=1$. We assume additionally
that $\sup_\xi \left|\chi_{q}(\xi)\right|=1$, for all
$q\in\mathbb{N}$. Let $\tau\geq0$, the wavelet $\psi$, given by
$\hat{\psi}(u)=e^{-\left|u\right|-(1/(\left|u\right|-\tau))}$ for 
$\left|u\right|>\tau$ and  $\hat{\psi}(u)=0$ for 
$\left|u\right|\leq\tau,$ has index of non-degenerateness $\tau$.
Consider the $C(\mathbb{R})$-valued distribution
$$\mathbf{f}(t,\xi)=\sum_{q=1}^{\infty}e^{q+i\left(\tau+\frac{1}{q}\right)t} \chi_{q}(\xi)\in\mathcal{S}'(\mathbb{R}_{t}, C(\mathbb{R}_{\xi}))\setminus
\mathcal{S}'(\mathbb{R}_{t}, C_{b}(\mathbb{R}_{\xi})) .$$
Then,
$$\mathcal{W}_{\psi}\mathbf{f}(x,y)(\xi)=\sum_{1\leq q<\frac{y}{\tau(1-y)}}e^{ q+(ix-y)\left(\tau+\frac{1}{ q}\right) -\frac{ q}{y- q\tau(1-y)}}\chi_{ q}(\xi) , \ \ \ 0<y<1.$$
and hence,
$\left\|\mathcal{W}_{\psi}\mathbf{f}(x,y)\right\|_{C_{b}(\mathbb{R})}\leq
1,$ for all $0<y<1.$ Therefore, the hypotheses of Theorem
\ref{wnwceth3} are fully satisfied, however,
$\mathbf{f}-\mathbf{G}\notin\mathcal{S}'(\mathbb{R},
C_{b}(\mathbb{R}))$, for any
$\mathbf{G}\in\mathcal{S}'(\mathbb{R}, C(\mathbb{R}))$ with
$\operatorname*{supp}\hat{\mathbf{G}}\subseteq[-\tau,\tau]$.
\end{example}

\bigskip

We obtain a local class estimate characterization of $\mathcal{S}'(\mathbb{R}^{n},E)$ if we combine Theorem \ref{wnwceth3} with exactly the same argument employed in the proof of Theorem \ref{wnwceth2}.

\begin{theorem}
\label{wnwceth4} Let $\mathbf{f}\in\mathcal{S}'(\mathbb{R}^{n},X)$, let $\varphi\in\mathcal{S}(\mathbb{R}^{n})$ be non-degenerate with index of non-degenerateness $\tau$, and let $\varphi_{0}\in\mathcal{S}(\mathbb{R}^{n})$ be such that $\hat{\varphi}_{0}(u)\neq 0$ for all $|u|\leq \tau$. Then, 
$\mathbf{f}\in\mathcal{S}'(\mathbb{R}^{n},E)$ if and only if the ensuing three conditions hold:
\begin{enumerate}
\item[(i)] $M_{\varphi}^{\mathbf{f}}(x,y)$ takes values in $E$ for almost
all $(x,y)\in\mathbb{H}^{n+1}$ and is measurable as an $E$-valued function.
\item[(ii)] There exist $k,l\in\mathbb{N}$ such that 
\begin{equation}\label{local ce eq}
\int_{0}^{1} y^{k}(1+|x|)^{-l}\left\|M_{\varphi}^\mathbf{f}(x,y)\right\|\mathrm{d}x\mathrm{d}y<\infty
\end{equation}
\item [(iii)] There is  $l\in\mathbb{N}$ such that $(\mathbf{f}\ast \varphi_{0})(x)\in E$ a.e. is measurable and condition \eqref{eq1exTce} holds.

\end{enumerate}
\end{theorem}

In particular, when $\int_{\mathbb{R}}\varphi(t)\mathrm{d}t\neq 0$ the third hypothesis in Theorem \ref{wnwceth4} becomes superfluous and we therefore have,
\begin{corollary}
\label{wnwcephi} Let $\mathbf{f}\in\mathcal{S}'(\mathbb{R}^{n},X)$
and let $\varphi\in\mathcal{S}(\mathbb{R}^{n})$ be such that
$\int_{\mathbb{R}^{n}}\varphi(t)\mathrm{d}t\neq 0$. Then,
$\mathcal{S}'(\mathbb{R}^{n},E)$ if and only if the conditions (i) and (ii) from Theorem \ref{wnwceth4} are satisfied.
\end{corollary}

\section{Strongly non-degenerate test functions} \label{wnwces}
A strengthened version of both Theorem \ref{wnwceth1} and Theorem \ref{wnwceth3}
holds if we restrict the non-degenerate test functions to those fulfilling the requirements of the following definition.
\begin{definition}
\label{wnwced2} Let $\varphi\in\mathcal{S}(\mathbb{R}^{n})$. We call $\varphi$
\emph{strongly} non-degenerate if there exist constants $N\in\mathbb{N}$,
$r>0$, and $C>0$ such that
\begin{equation}
\label{wnwceeq13}
C\left|u\right|^{N}\leq |\hat{\varphi}(u)|\ , \ \ \ \mbox{for all
}\left|u\right|\leq r.
\end{equation}
\end{definition}

\bigskip

The class of test functions from Definition \ref{wnwced2} turns out to be the same as that employed by Drozhzhinov and Zav'yalov in \cite{drozhzhinov-z3}, but they formulated their notion of non-degenerateness in terms of the Taylor polynomials of $\hat{\varphi}$. We say that a polynomial $P$ is non-degenerate (at the origin, in the Drozhzhinov-Zav'yalov sense) if
for each $\omega\in\mathbb{S}^{n-1}$ one has that
$$
P(r\omega)\not\equiv 0, \ \ \ r\in\mathbb{R}_{+}.$$
Drozhzhinov and Zavialov have then considered the class of test functions $\varphi\in\mathcal{S}(\mathbb{R}^{n})$
for which there exists $N\in\mathbb{N}$ such that
$$\sum_{\left|m\right|\leq N}\frac{\hat{\varphi}^{(m)}(0)u^{m}}{m!}$$ is a non-degenerate polynomial. One readily verifies that the latter property introduced by Drozhzhinov and Zav'ylalov is equivalent to strong non-degenerateness in the sense of Definition \ref{wnwced2}. It should also be noticed that strong non-degenerateness is included in Definition \ref{wnwd1}, but, naturally, Definition \ref{wnwd1} gives much more test functions (cf. \cite[Remark 4.3, p. 9]{Pil-Vin2014}). 

We can now state our first result concerning strongly non-degenerate test functions. 

\begin{theorem}
\label{wnwceth5} Let $\mathbf{f}\in\mathcal{S}'(\mathbb{R}^{n},X)$
and let $\varphi\in\mathcal{S}(\mathbb{R}^{n})$ be
strongly non-degenerate. Assume:
\begin{enumerate}
\item[(i)] $M_{\varphi}^{\mathbf{f}}(x,y)$ takes values in $E$
for almost all $(x,y)\in\mathbb{R}^{n}\times(0,1]$ and is measurable as an $E$-valued function on $\mathbb{R}^{n}\times(0,1]$.
\item[(ii)] There are $k,l\in\mathbb{N}$ such that \eqref{local ce eq} holds.
\end{enumerate}
Then, there exists $\mathbf{G}\in \mathcal{S}'(\mathbb{R}^{n},X)$ such
that $\mathbf{f}-\mathbf{G}\in\mathcal{S}'(\mathbb{R}^{n},E)$ and
$\operatorname*{supp} \hat{\mathbf{G}}\subseteq\left\{0\right\}$. Moreover, the relations \eqref{wnwceeq10} hold.

If, in addition, the condition (iii) from Theorem \ref{wnwceth4} holds for some test function $\varphi_{0}\in \mathcal{S}(\mathbb{R}^{n})$ such that $\int_{\mathbb{R}^{n}} \varphi_{0}(t)\mathrm{d}t\neq 0$, then $\mathbf{f}\in\mathcal{S}'(\mathbb{R}^{n},E)$
\end{theorem}
\begin{proof}
By Theorem \ref{wnwceth3}, we may assume that
$\operatorname*{supp}\hat{\mathbf{f}} \subseteq\overline{B(0,1)}$.
 Let $\rho\in
\mathcal{S}(\mathbb{R}^{n})$ such that $\rho(u)=1$ for $u\in
B(0,3/2)$ and $\operatorname*{supp}\rho\subset B(0,2)$. We can find $\sigma,C_1>0$ and
$N\in\mathbb{N}$ such that $2\sigma\leq 1$ and
$C_1|u|^{N}\leq|\hat{{\varphi}}(u)|$, for all $u\in
\overline{B(0,2\sigma)}$. Given
$\hat{\eta}\in\mathcal{S}_{0}(\mathbb{R}^{n})$, then
$
\hat{\chi}(u)=\hat\chi_{\hat{\eta}}(u)=\rho(u)\eta(u)/\hat{\varphi}(\sigma
u)
$ defines an element of $\mathcal{S}(\mathbb{R}^{n})$ in a
continuous fashion; consequently, the mapping
$\gamma:\mathcal{S}_{0}(\mathbb{R}^{n})\to [0,\infty)$ given
by
$\gamma(\hat{\eta})=(2\pi)^{n}\int_{\mathbb{R}^{n}}(1+\left|\xi\right|)^{l}\left|\chi(\xi)\right|\mathrm{d}\xi$
is a continuous seminorm over $\mathcal{S}_{0}(\mathbb{R}^{n})$.
Now, for any $\hat{\eta}\in\mathcal{S}_{0}(\mathbb{R}^{n})$,
\begin{align*}
\left\langle \mathbf{f},\hat{\eta}\right\rangle & =\left\langle \hat{\mathbf{f}}(u),\hat{\chi}(u)\hat{\varphi}(\sigma u)\right\rangle
=(2\pi)^{n}\int_{\mathbb{R}^{n}}\chi(-\xi)M_{\varphi}^\mathbf{f}(\xi,\sigma)\mathrm{d}\xi.
\end{align*}
Therefore, $\left\|\left\langle \mathbf{f},\eta\right\rangle\right\|\leq (C/\sigma^{k})\gamma(\hat{\eta}),$ for all $\hat{\eta}\in\mathcal{S}_{0}(\mathbb{R}^{n}),$
and the latter implies that the restriction of $\mathbf{f}$ to $\mathcal{S}_{0}(\mathbb{R}^{n})$ belongs to $\mathcal{S}'_{0}(\mathbb{R}^{n},E)$. The argument we already gave in the proof of Theorem \ref{wnwceth1} yields the existence of $\mathbf{G}$ satisfying all the requirements. That  \eqref{wnwceeq10} holds for each $q$ is shown exactly as in the proof of Theorem \ref{wnwceth1}. Finally, the proof of the last assertion of the theorem is the same as that of Theorem \ref{wnwceth2}.

\end{proof}

In dimension $n=1$, there is no distinction between non-degenerateness and strong non-degenerateness, whenever we consider test functions from $\mathcal{S}(\mathbb{R})\setminus\mathcal{S}_{0}(\mathbb{R})$. Actually, a stronger result than Theorem \ref{wnwceth5} holds in the one-dimensional case.

\begin{proposition}
\label{wnwcep3}
Let $\mathbf{f}\in\mathcal{S}'(\mathbb{R},X)$
and let $\varphi\in\mathcal{S}(\mathbb{R})$ be such that $\int_{-\infty}^{\infty}t^{d}\varphi(t)\neq0$, for some $d\in\mathbb{N}$. If the conditions (i) and (ii) of Theorem \ref{wnwceth5} are satisfied, then there exists an $X$-valued polynomial $\mathbf{P}$ of degree at most $d-1$ such that $\mathbf{f}-\mathbf{P}\in\mathcal{S}'(\mathbb{R},E)$. 
\end{proposition}
\begin{proof} We assume that $d$ is the least integer with the assumed property. There exists $\phi\in\mathcal{S}(\mathbb{R})$ such that $\phi^{(d)}=(-1)^{d}\varphi$, and $\int_{-\infty}^{\infty}\phi(t)\neq0$. Then,
$
M_{\phi}^{\mathbf{f}^{(d)}}(x,y)=y^{-d}M_{\varphi}^{\mathbf{f}}(x,y).
$
Hence, an application of Corollary \ref{wnwcephi} gives that $\mathbf{f}^{(d)}\in\mathcal{S}'(\mathbb{R},E)$, and this clearly implies the existence of $\mathbf{P}$ with the desired properties.
\end{proof}

Observe that the conclusion of Proposition \ref{wnwcep3} does not hold for multidimensional regularizing transforms, in general, even if the kernels $\varphi$ are strongly non-degenerate. This fact is shown by Example \ref{wnwceex1}. Naturally, if $X$ is a normed space in Theorem \ref{wnwceth5}, then $\mathbf{G}$ must be an $X$-valued polynomial, this fact is stated in the next corollary. Corollary \ref{wnwcec2} extends an important result of Drozhzhinov and Zavialov \cite[Thm. 2.1]{drozhzhinov-z3}.

\begin{corollary}
\label{wnwcec3} Let the hypotheses of Theorem \ref{wnwceth5} be satisfied. If $X$ is a normed space, then there is an
$X$-valued polynomial $\mathbf{P}$ such that
$\mathbf{f-P}\in\mathcal{S}'(\mathbb{R}^{n},E)$.  Furthermore, $\mathbf{f}\in \mathcal{S}'_{I_{\varphi}}(\mathbb{R}^{n}, E)$.
\end{corollary}

As in Corollary \ref{wnwcec2}, we recover the following result of Drozhzhinov and Zav'yalov  (cf. \cite[Thm. 2.2]{drozhzhinov-z3}
).

\begin{corollary}
\label{wnwcec4} Let the hypotheses of Corollary \ref{wnwcec3} be satisfied. If there is $d\in\mathbb{N}$ such that $\mathbb{P}_{d}(\mathbb{R}^{n})$ is contained in the ideal generated by the polynomials $P_{0}^{\varphi},P_1^{\varphi},P_2^{\varphi},\dots,P_{d}^{\varphi}$, then there exists an $X$-valued polynomial $\mathbf{P}$ of degree at most $d-1$ such that $\mathbf{f}-\mathbf{P}\in\mathcal{S}'(\mathbb{R}^{n},E)$. In particular, $\mathbf{f}\in\mathcal{S}'_{\mathbb{P}_{d}}(\mathbb{R}^{n},E)$.
\end{corollary}

\section{Vector-valued distributions intertwining representations of $\mathbb{R}^{n}$}
\label{section class estimates with representation}
As an application of our ideas, we extend the one-dimensional considerations from \cite[Sect. 4]{drozhzhinov-z2} to the multidimensional case. 
Throughout this section we suppose that the Banach space $E$ is continuously and linearly included into the locally convex space $X$ and that both carry a representation of $(\mathbb{R}^{n},+)$, that is, there is $\pi: \mathbb{R}^{n}\to L_{b}(X)$ such that 
\begin{enumerate}
\item [(a)]  $\pi(x+h)= \pi(x)\pi(h)$, for all $x,h\in\mathbb{R}^{n}$.
\item [(b)] $\pi(x)\mathbf{v}\in E$ for every $\mathbf{v}\in E$ and $x\in\mathbb{R}^{n}$.
\end{enumerate}
It follows from (b) and the closed graph theorem that in fact (the restriction of) $\pi(x)\in L_{b}(E)$ and thus $\pi$ induces a representation of $\mathbb{R}^{n}$ on $E$ as well. We further assume that $\pi$ is a tempered $C_0$-group of operators on $E$, namely, 
\begin{enumerate}
\item [(c)] There are $l$ and $C$ such that
$
\|\pi(x)\|_{L_{b}(E)}\leq C(1+|x|)^{l},$ $ x\in\mathbb{R}^{n}.$
\item [(d)] $\displaystyle \lim_{x\to 0} \|\pi(x)\mathbf{v}-\mathbf{v}\|=0$ for each $\mathbf{v}\in E$.
\end{enumerate}

We denote the translation operators on $\mathbb{R}^{n}$ as $T_{h}$, so that their actions on functions (and vector-valued distributions) are given by $(T_{h}\phi)(x)=\phi(x-h)$. 
We then say that $\mathbf{f}\in \mathcal{S}'(\mathbb{R}^{n}, X)$
 intertwines $\pi$ and $T$ if  $\pi(x)\circ\mathbf{f}=T_{-x}\mathbf{f}(=\mathbf{f}\circ T_{x})$ for each $x\in\mathbb{R}^{n}$, that is, if
$$
\pi(x)\left(\langle \mathbf{f},\varphi\rangle\right)= \langle \mathbf{f}, T_{x}\varphi\rangle, \qquad \mbox{for every } \varphi\in\mathcal{S}(\mathbb{R}^{n}) \mbox{ and }x\in\mathbb{R}^{n}.
$$

We shall also need the notion of regularly varying functionals, introduced and studied by Drozhzhinov and Zav'yalov in \cite[Sect. 2]{drozhzhinov-z2}. Let $\mathbb{J}$ be a non-negative functional acting on non-negative measurable functions $g:(0,1]\to [0,\infty]$. It is notational convenient to employ a dummy variable of evaluation and write $\mathbb{J}(g)=\mathbb{J}_{y}(g(y))$. The functional is called regularly varying of index $\alpha$ if the ensuing five conditions are satisfied:
\begin{enumerate}
\item [(I)] $\mathbb{J}_{y}(\int g(y,\xi)\mathrm{d}\xi)\leq \int \mathbb{J}_{y}(g(y,\xi))\mathrm{d}\xi $ for all non-negative measurable function $g(y,\xi)$.
\item [(II)] $\mathbb{J}$ is monotone, $\mathbb{J}(g_{1})\leq \mathbb{J}(g_{2})$ whenever $g_{1}(y)\leq g_{2}(y)$ a.e.
\item [(III)] $\mathbb{J}$ is homogeneous, $\mathbb{J}(\lambda g)=\lambda \mathbb{J}(g)$, $\lambda\geq 0$.
\item [(IV)] $\mathbb{J}$ has the monotone convergence property, that is, 
$\mathbb{J}(g_{k})\to \mathbb{J}(g)\to$ as $k\to\infty$, whenever $g_k(y) \nearrow g(y)$ a.e. as $k\to\infty$.
\item [(V)] For every $\varepsilon>0$ there is $C_{\varepsilon}>0$ such that 
$$
\mathbb{J}_{y}(g (ay))\leq 
\begin{cases} 
C_{\varepsilon} a^{\alpha+\varepsilon}\mathbb{J}_{y}(g (y)) & \quad \mbox{if }a\geq 1\\
C_{\varepsilon} a^{\alpha-\varepsilon}\mathbb{J}_{y}(g (y)) & \quad \mbox{if }a\leq 1,
\end{cases}
$$
for every non-negative measurable function with support in $(0,1]$.
\end{enumerate}

\begin{example}
\label{ex regularly varying functional}
A typical example of such a $\mathbb{J}$ is given by a weighted $L^{q}$ norm ($1\leq q\leq\infty$) with respect to a Karamata regularly varying function. Let $c\in L^{\infty}_{loc}(0,1]$ be regularly varying at 0 of index $\alpha$, that is, a positive measurable function that satisfies
$$
\lim_{y\to 0^{+}} \frac{c(ay)}{c(y)}=a^{\alpha}, \qquad a>0.
$$
In view of Potter's estimates \cite[p. 25]{b-g-t}, the functional 
$$\mathbb{J}^{q,c}(g)= \left(\int_{0}^{1} \left(\frac{g(y)}{c(y)} \right)^{q}\frac{\mathrm{d}y}{y} \right)^{\frac{1}{q}},$$
with obvious adjustments when $p=\infty$, defines a regularly varying functional of index $\alpha$. When $c(y)=y^{\alpha}$, we simply write $\mathbb{J}^{q,c}=\mathbb{J}^{q,\alpha}$.
\end{example}

\bigskip

 We also need the ensuing definition.

\begin{definition} \label{def LP pair} Let $ \alpha\in\mathbb{R}$, $\varphi_{0} \in\mathcal{S}(\mathbb{R}^{n})$, and let $ \varphi \in {\mathcal S} (\mathbb{R}^n) $ be a non-degenerate test function with index of non-degenerateness $ \tau.$ The pair $(\varphi_0,\varphi)$ is said to be a \emph{Littlewood-Paley pair} (LP-pair) of order $\alpha$ if $ \hat{\varphi}_{0} (u) \neq 0 $ for $|u|\leq \tau $ and $\varphi\in \mathcal{S}_{\mathbb{P}_{\lfloor \alpha \rfloor}}(\mathbb{R}^{n})$, i.e., 
$\int_{\mathbb{R}^{n}}t^{m}\varphi(t)\mathrm{d}t = 0$ for all multi-index $m\in\mathbb{N}^{n}$ with $\left| m \right|\leq \lfloor\alpha\rfloor $. (Note that if $\alpha<0$, the latter condition on the moments of $\varphi$ is empty and thus simply becomes $
\varphi\in \mathcal{S}(\mathbb{R}^{n})$.)
\end{definition}

We also point out that \cite[Lemma 2.5, p. 711]{drozhzhinov-z2} every regularly varying functional $\mathbb{J}$ of index $\alpha$ satisfies inequalities
\begin{equation}
\label{ineq rvf}
\mathbb{J}^{1,\beta}(g) \leq C_{\beta} \mathbb{J}(g),
\end{equation}
 for some $C_{\beta}>0$ if $\beta<\alpha$.

\begin{theorem} \label{th1 LP}Let $\mathbf{f}\in\mathcal{S}'(\mathbb{R}^{n},X)$ intertwine the representation $\pi$ and the translation group $T$, let $\mathbb{J}$ be a regularly varying functional of index $\alpha\in\mathbb{R}$, and let $(\varphi_0,\varphi)$ be an LP-pair of order $\alpha$. Assume that $\langle \mathbf{f},\varphi_{0} \rangle\in E$, $M_{\varphi}^{\mathbf{f}}(0,y)\in E$
for almost all $y\in(0,1]$ and is measurable as an $E$-valued function on $(0,1]$, and
\begin{equation}
\label{gBesov inequality 1}
\mathbb{J}_{y}(\|M_{\varphi}^{\mathbf{f}}(0,y)\|)<\infty.
\end{equation}
Then, $\mathbf{f}\in\mathcal{S}'(\mathbb{R}^{n},E)$ and there is a continuous seminorm $\gamma$ on $\mathcal{S}(\mathbb{R}^{n})$, independent of $\mathbf{f}$, such that
\begin{equation}
\label{gBesov inequality 2}
\|  \langle \mathbf{f}, \rho \rangle \|+\mathbb{J}_{y}(\|M^{\mathbf{f}}_{\theta}(0,y)\|) \leq (\|  \langle \mathbf{f}, \varphi_{0} \rangle \|+\mathbb{J}_{y}(\|M^{\mathbf{f}}_{\varphi}(0,y))\|)(\gamma(\rho)+\gamma(\theta)),
\end{equation}
for every $\rho\in \mathcal{S}(\mathbb{R}^{n})$ and $\theta\in\mathcal{S}_{\mathbb{P}_{\lfloor \alpha \rfloor}}(\mathbb{R}^{n})$.
\end{theorem}
\begin{proof} Note that

$$
M^{\mathbf{f}}_{\varphi}(x,y)=\pi (x) (M^{\mathbf{f}}_{\varphi}(0,y)),
$$
whence one checks that the condition (i) of Theorem \ref{wnwceth4} is satisfied, as implied by the assumption (d) on $\pi$. Using \eqref{gBesov inequality 1}, the property (c) of $\pi$, and \eqref{ineq rvf}, one obtains that the hypothesis (ii) of Theorem \ref{wnwceth4} holds. Furthermore, the function $\mathbf{f}\ast \check{\varphi}_{0}\in C(\mathbb{R}^{n},E)$ and satisfies condition (iii) of Theorem \ref{wnwceth4}. Consequently, $\mathbf{f}\in\mathcal{S}'(\mathbb{R}^{n},E)$. 

In order to find $\gamma$ such that \eqref{gBesov inequality 2} holds, we first consider 
$$
L^{\mathbb{J}}((0,1],E)=\left\{ \mathbf{g}:(0,1]\to E:\: \|\mathbf{g}(\lambda)\|_{E} \mbox{ is measurable and } \mathbb{J}_{\lambda}(\|\mathbf{g}(\lambda)\|_{E})<\infty \right\},
$$
and analogously for other regularly varying functionals; it is a Banach space  (cf. \cite[Remark 2.1, p. 712]{drozhzhinov-z2}). Obviously, there is $\beta\in\mathbb{R}$ such that 
$$
\int_{0}^{1} \|\langle \mathbf{f}(\lambda t), \phi(t) \rangle\|\frac{\mathrm{d}\lambda}{\lambda^{\beta+1}}<\infty \qquad \mbox{for each }\phi\in\mathcal{S}(\mathbb{R}^{n}).
$$
This means that the vector-valued distribution $\mathbf{F}$ given by
$$
\langle \mathbf{F}(t), \phi (t)\rangle (\lambda) = \langle \mathbf{f}(\lambda t), \phi(t) \rangle
$$
takes values in the Banach space $L^{\mathbb{J}^{1,\beta}}((0,1],E)$, with $\mathbb{J}^{1,\beta}$ as defined in Example \ref{ex regularly varying functional}. We may assume that $\beta<\alpha$ so that $L^{\mathbb{J}}((0,1],E)$ is continuously included into $L^{\mathbb{J}^{1,\beta}}((0,1],E)$, as follows from \eqref{ineq rvf}. Since
\begin{align*}
\| M^{\mathbf{F}}_{\varphi}(x,y)\|_{L^{\mathbb{J}}((0,1],E)}&=\mathbb{J}_{\lambda}(\|M_{\varphi}^{\mathbf{f}}(\lambda x, \lambda y)\|)
\\
&
\leq (1+|x|)^{l}\mathbb{J}_{\lambda}(\|M_{\varphi}^{\mathbf{f}}(0,\lambda y)\|) 
\\
&
\leq C_{\varepsilon}y^{\alpha}\left(y^{\varepsilon}+\frac{1}{y^{\varepsilon}}\right) (1+|x|)^{l}\mathbb{J}_{\lambda}(\|M_{\varphi}^{\mathbf{f}}(0,\lambda)\|), \qquad (x,y)\in \mathbb{H}^{n},
\end{align*}
Corollary applies to conclude the existence of functions $\mathbf{c}_{m}\in L^{\mathbb{J}^{1,\beta}}((0,1],E),$ $|m|\leq N$,
such that
\begin{equation}
\label{besov eq 3}
\mathbb{J}_{\lambda}\Big(\Big\|\langle \mathbf{f}(\lambda t), \phi(t) \rangle - \sum_{|m|\leq N}\mathbf{c}_{m}(\lambda) \int_{\mathbb{R}^{n}} t^{m}\phi(t)\mathrm{d}t\Big\|  \Big)<\infty, \qquad \phi\in\mathcal{S}(\mathbb{R}^{n}).
\end{equation}
We may also assume that each $\mathbf{c}_{m}$ is bounded on $[1/2,1]$. Fix $|m|\leq N$ and pick $\phi\in\mathcal{S}(\mathbb{R}^{n})$ with $\int_{\mathbb{R}^{n}}t^{j}\phi(t)\mathrm{d}t=\delta_{j,m}$. Then, for each fixed $a>0$
\begin{align*}
\mathbf{c}_{m}(a\lambda)- a^{|m|}\mathbf{c}_{m}(\lambda)
&= \Big(\langle\mathbf{f}(\lambda t), a^{-n}\phi(t/a) \rangle - \sum_{|m|\leq N}\mathbf{c}_{m}(\lambda) a^{-n}\int_{\mathbb{R}^{n}} t^{m}\phi(t/a)\mathrm{d}t\Big)
\\
&
\quad - \Big(\langle\mathbf{f}(a\lambda t),  \phi(t) \rangle
 - \sum_{|m|\leq N}\mathbf{c}_{m}(a\lambda) \int_{\mathbb{R}^{n}} t^{m}\phi(t)\mathrm{d}t\Big) \in L^{\mathbb{J}}((0,1],E).
\end{align*}
We use these relations with a fixed multi-index $\alpha \leq |m|\leq N$ and $a=1/2$, and write
\begin{equation}
\label{besov eq 4}
\mathbf{c}_{m}(\lambda/2)- 2^{-|m|}\mathbf{c}_{m}(\lambda)=\mathbf{b}(\lambda)\in L^{\mathbb{J}}((0,1],E).
\end{equation}
Iterating \eqref{besov eq 4} $\nu$ times, one deduces that for each $\nu\in \mathbb{Z}_{+}$
$$
\mathbf{c}_{m}(\lambda)= 2^{|m|}\left(2^{-\nu |m|} \mathbf{c}_{m}(2^{\nu}\lambda)+ \sum_{j=1}^{\nu} 2^{-|m|j} \mathbf{b}(2^{j}\lambda)\right), \qquad 0<\lambda\leq 2^{-\nu}.
$$
Denote as $\chi_{A}$ the characteristic function of a set $A$ and pick $0<\varepsilon< |m|-\alpha$. Then,
\begin{align*}
\mathbb{J}(\|\mathbf{c}_{m}\|)&\leq \mathbb{J}(\chi_{(1/2,1]}\|\mathbf{c}_{m}\|)
+ 2^{|m|}\sum_{\nu=1}^{\infty}2^{-\nu|m|}\mathbb{J}_{\lambda}(\chi_{(1/2,1]}(2^{\nu}\lambda)\|\mathbf{c}_{m}(2^{\nu}\lambda)\|)
\\
&
\qquad + 2^{|m|}\mathbb{J}_{\lambda}\left(\Big\|\sum_{j=1}^{\infty}2^{-j|m|}\mathbf{b}(2^{j}\lambda)\chi_{(0,2^{-j}]}(\lambda)\Big\|\right)
\\
&
\leq C_{\varepsilon} \left(\mathbb{J}(\chi_{(1/2,1]}\|\mathbf{c}_{m}\|) +\mathbb{J}(\chi_{(0,1/2]}\|\mathbf{b}\|)\right )\sum_{j=1}^{\infty} 2^{-j(m-\alpha-\varepsilon)}<\infty,
\end{align*}
so that
\begin{equation}
\label{besov eq 5}
\mathbf{c}_{m}\in L^{\mathbb{J}}((0,1],E) \qquad\mbox{for } \alpha< |m|\leq N.
\end{equation}
Combining \eqref{besov eq 3} with \eqref{besov eq 5}, we now conclude that 
\begin{equation}
\label{besov eq 6}
\mathbb{J}_{y}\left(\|M_{\phi}^{\mathbf{f}}(0,y)\|\right)<\infty, \qquad \mbox{for every } \phi\in\mathcal{S}_{\mathbb{P}_{\lfloor \alpha \rfloor}}(\mathbb{R}^{n}).
\end{equation}

Our final step is to use (\ref{besov eq 6}) to infer the existence of the sought seminorm $\gamma$. We define two normed spaces of $E$-valued distributions intertwining $\pi$ and $T$. The first of them is the space $Y$  consisting of all those $\mathbf{g}$ such that 
\begin{equation}
\label{besov eq 7}
\|\mathbf{g}\|_{Y}= \|\langle \mathbf{g}, \varphi_0 \rangle\|+ \mathbb{J}_{y} (M_{\varphi}^{\mathbf{g}}(0,y))<\infty.
\end{equation}
For the second space, we consider a fixed but arbitrary bounded set $\mathfrak{B}\subset \mathcal{S}(\mathbb{R}^{n})\times \mathcal{S}_{\mathbb{P}_{\lfloor \alpha \rfloor}}(\mathbb{R}^{n})$ with the only requirement that $\mathfrak{B}$ contains the $LP$-pair $(\varphi_0,\varphi)$ and define $\tilde{Y}$ as the space of those $\mathbf{g}$ such that 
\begin{equation}
\label{besov eq 8}
\|\mathbf{g}\|_{\tilde{Y}}= \sup_{(\rho,\theta)\in \mathfrak{B}}(\|\langle \mathbf{g}, \rho \rangle\|+ \mathbb{J}_{y} (M_{\theta}^{\mathbf{g}}(0,y)))<\infty.
\end{equation}
A standard argument (see, e.g., \cite[Assertion 6.2, p. 740]{drozhzhinov-z2} or \cite[Prop. 5.4, p. 15]{P-R-V}) shows that both $Y$ and $\tilde{Y}$ are Banach spaces. The relation (\ref{besov eq 6}) applied to an arbitrary element of $Y$ and the Banach-Steinhaus theorem implies that $Y=\tilde{Y}$ as vector spaces. Since the identity mapping $(Y,\|\:\cdot\|_{\tilde{Y}})\to (Y,\|\:\cdot\|_{Y})$ is obviously continuous, the open mapping theorem gives that the norms \eqref{besov eq 7} and \eqref{besov eq 8} are equivalent. Thus, applying again the Banach-Steinhaus theorem, we get that the bilinear mapping
$$
Y\times (\mathcal{S}(\mathbb{R}^{n})\times \mathcal{S}_{\mathbb{P}_{\lfloor \alpha \rfloor}}(\mathbb{R}^{n})) \ni (\mathbf{g}, (\rho,\theta))\mapsto (\langle \mathbf{g},\rho \rangle, M_{\theta}^{\mathbf{g}}(0, \: \cdot \:) )\in  E\times L^{\mathbb{J}}((0,1],E).
$$
is continuous. This yields the existence of $\gamma$ with the desired properties. The proof is complete.

\end{proof}
\begin{example}[\emph{Besov spaces}]
\label{ex Besov} Let $(\varphi_0,\varphi)$ be an LP-pair of order $s$, let $p,q\in[1,\infty]$, and let $c\in L^{\infty}_{loc}(0,1]$ be a regularly varying function (at 0) of index $s$. We define the Besov space $B^{c}_{p,q}(\mathbb{R}^{n})$ as the Banach space of all $f\in\mathcal{S}'(\mathbb{R}^{n})$ such that\footnote{When $p=\infty$, we apply our considerations with $E=UC(\mathbb{R}^{n})$, the space of uniformly continuous functions; the $L^{\infty}$ assumption itself implies \cite[Thm. 3]{D-P-V2015TIB} $f\ast \varphi_0\in UC(\mathbb{R}^{n})$ and $M^{f}_{\varphi}(\:\cdot\:, y) \in UC(\mathbb{R}^{n})$ for each $y>0$.} $f\ast \varphi_0\in L^{p}(\mathbb{R}^{n})$ and $M^{f}_{\varphi}(\:\cdot\:, y) \in L^{p}(\mathbb{R}^{n})$ for each $(0,1]$ and 
$$
\|f\|_{B^{c}_{p,q}(\mathbb{R}^{n})}= \|f\ast \varphi_0\|_{L^{p}(\mathbb{R}^{n})}| + \mathbb{J}^{q,c}_{y}(\|M^{f}_{\varphi}(x, y) \|_{L^{p}(\mathbb{R}_{x}^{n})}) < \infty.
$$
When $c(y)=y^{s}$, we recover the classical Besov space $B^{s}_{p,q}(\mathbb{R}^{n})=B^{c}_{p,q}(\mathbb{R}^{n})$. 

If we consider $E=L^{p}(\mathbb{R}^{n})$ and the vector-valued distribution $\mathbf{f}$ given by $\langle \mathbf{f}, \rho \rangle= f\ast \check{\rho}$, which takes values in $X=L^{p}(\mathbb{R}^{n}, (1+|x|)^{N}\mathrm{d}x)$ for some sufficiently large $N$, Theorem \ref{th1 LP} yields that the definition of $B^{c}_{p,q}(\mathbb{R}^{n})$ is independent of $(\varphi_{0},\varphi)$ and that different choices of LP-pairs of order $s$ lead to equivalent norms. Moreover, given arbitrary $\rho\in \mathcal{S}(\mathbb{R}^{n})$ and $\theta\in  \mathcal{S}_{\mathbb{P}_{\lfloor \alpha \rfloor}}(\mathbb{R}^{n})$, there are contants $C_1$ and $C_2$ such that 
\begin{equation}
\label{besov eq1}
\|f\ast \rho \|_{L^{p}(\mathbb{R}^{n})}\leq C_1\|f\|_{B^{c}_{p,q}(\mathbb{R}^{n})}
\end{equation}
and 
\begin{equation}
\label{besov eq2}
\mathbb{J}^{q,c}_{y}(\|M^{f}_{\theta}(x, y) \|_{L^{p}(\mathbb{R}_{x}^{n})})\leq C_2\|f\|_{B^{c}_{p,q}(\mathbb{R}^{n})}
\end{equation} 
for all $f\in B^{c}_{p,q}(\mathbb{R}^{n})$. The constants $C_1$ and $C_2$ in these inequalities can be chosen to be the same if we let $\rho$ and $\theta$ vary on bounded subsets of $\mathcal{S}(\mathbb{R}^{n})$. If $s<0$, we can take here any $\theta\in \mathcal{S}(\mathbb{R}^{n})$ without any assumption on its moments.
\end{example}

\end{document}